\pdfoutput=1

\documentclass[letterpaper, 11pt]{amsart}
\usepackage{amsmath, amsfonts, amssymb, amsthm, amsrefs, array, stmaryrd, graphicx, hyperref, mathrsfs, eucal, caption, soul}

\usepackage[hhmmss]{datetime}

\allowdisplaybreaks

\AtBeginDocument{%
\def\MR#1{}
}  

\usepackage{color}

\newcommand{\kibitz}[2]{\ifnum\Comments=1\textcolor{#1}{#2}\fi}

\theoremstyle{plain}
\newtheorem{thm}{Theorem}[section]
\newtheorem{lemma}{Lemma}[section]
\newtheorem{prop}[lemma]{Proposition}

\theoremstyle{definition}

\theoremstyle{remark}
\newtheorem{remark}[lemma]{Remark}
\numberwithin{equation}{section}

\def\esmath{\ensuremath}

\def\phv{\ensuremath\varphi}

\def\RR{\esmath\mathbb R} 

\newcommand{\p}{\partial}

\newcommand{\bn}{\begin{enumerate}}
\newcommand{\en}{\end{enumerate}}
\newcommand{\bi}{\begin{itemize}}
\newcommand{\ei}{\end{itemize}}
\newcommand{\bqq}{\begin{eqnarray*}}
\newcommand{\eqq}{\end{eqnarray*}}
\newcommand{\balg}{\begin{align*}}
\newcommand{\ealg}{\end{align*}}

\DeclareMathOperator{\Rm}{Rm}

\begin{document}


\title[Noncompact MCF with Type-II curvature blow-up. II]{Mean curvature flow of noncompact hypersurfaces with Type-II curvature blow-up. II}

\author{James Isenberg}
\address{Department of Mathematics, University of Oregon, Eugene, OR 97403, USA}
\email{isenberg@uoregon.edu}

\author{Haotian Wu}
\address{School of Mathematics and Statistics, The University of Sydney, NSW 2006, Australia}
\email{haotian.wu@sydney.edu.au}

\author{Zhou Zhang}
\address{School of Mathematics and Statistics, The University of Sydney, NSW 2006, Australia}
\email{zhou.zhang@sydney.edu.au}

\date{\usdate\today} 




\begin{abstract}
We continue the study, initiated by the first two authors in \cite{IW19}, of Type-II curvature blow-up in mean curvature flow of complete noncompact embedded hypersurfaces. In particular, we construct mean curvature flow solutions, in the rotationally symmetric class, with the following precise asymptotics near the ``vanishing'' time $T$: (1) The highest curvature concentrates at the tip of the hypersurface (an umbilical point) and blows up at the rate $(T-t)^{-1}$. (2) In a neighbourhood of the tip, the solution converges to a translating soliton known as the bowl soliton. (3) Near spatial infinity, the hypersurface approaches a collapsing cylinder at an exponential rate.
\end{abstract}

\maketitle

\section{Introduction}\label{sec:intro}

This paper continues the investigation by the first two authors \cite{IW19} concerning the Type-II curvature blow-up in mean curvature flow (MCF) of \emph{noncompact} hypersurfaces embedded in Euclidean space.

Given a one-parameter family of embeddings (or more generally, immersions) $\phv(t):M^n\to\mathbb{R}^{n+1}$, $t_0<t<t_1$, of $n$-dimensional hypersurfaces in the Euclidean space, MCF is defined by the following evolution equation
\begin{align}\label{eq:mcf}
\partial_t\phv (p,t) = \vec H,\quad p\in M^n,\quad t_0\leqslant t<t_1.
\end{align}
which geometrically deforms the hypersurface in the direction of its mean curvature vector $\vec H$.

In local coordinates, the MCF equation \eqref{eq:mcf} is a (weakly) parabolic PDE system whose short-time existence and uniqueness is well-known. Although the flow has smoothing property in short time, it can develop singularities over larger time scales for many initial data. For example, under MCF and in finite time, any closed convex hypersurface develops a ``spherical singularity'' \cite{Hui84}, whereas hypersurfaces close to a round cylinder develop a ``cylindrical singularity'' \cite{GS09}.

Let $M_t:=\phv(t)(M^n)$ be the hypersurface under MCF at time $t$ and $h(p,t)$ the second fundamental form of $M_t$ at $p$. Suppose MCF of $M_t$ becomes singular at time $t=T<\infty$. Then this finite-time singularity is called \emph{Type-I} if
\begin{align*}
\sup\limits_{p\in M_t} \vert h(p,t)\vert(T-t)^{1/2} \leqslant C
\end{align*}
for some finite constant $C$, and it is called \emph{Type-II} (more precisely, Type-IIa\footnote{A MCF solution is said to be \emph{Type-IIb} if it exists for $t\in[t_0,\infty)$ and $\sup_{M_t}\vert h(\cdot,t)\vert$ blows up at a rate faster than $Ct^{-1/2}$ for some constant $C$. We study the precise asymptotics of Type-IIb MCF solutions elsewhere \cite{IWZ19b}.}) if $\sup_{M_t}\vert h(\cdot,t)\vert$ blows up at a faster rate.

Examples of Type-I MCF solutions are plentiful in all dimensions. For example, MCF of any closed embedded curve in the plane always becomes convex \cite{Gray87} and then forms a Type-I round singularity \cite{GH86}. In dimension two or higher, typical Type-I examples include the round sphere, the round cylinder, and hypersurfaces in suitable open sets around them \cite{Hui84,GS09}. In contrast, MCFs which develop Type-II singularities are more difficult to specify and are typically expected to appear if the behaviour of the flow undergoes a ``phase change''. To explain what we mean by such a phase change, we consider the following scenarios.

Consider a one-parameter family of rotationally symmetric n-spheres ($n\geqslant 2$) embedded in $\mathbb{R}^{n+1}$ with the parameter controlling the extent to which the equator is tightly cinched. Depending on the amount of cinching, MCF starting from a 2-sphere in this family has the following behaviours:
(i) For very loose cinching, the flow converges to the shrinking round sphere with its usual (global) Type I singularity \cite{Hui84}.  (ii) For very tight cinching, the equator shrinks more rapidly than the two ``dumbbell'' hemispheres,  and forms a (local) Type-I ``neckpinch'' modelled locally by a cylinder \cite{Hui90}. Scenarios (i) and (ii) represent different behaviours of MCF. As the cinching parameter varies from ``very loose'' to ``very tight'', we expect: (iii) at some ``threshold'' parameter in between, MCF forms a finite-time singularity that is not Type-I, and hence Type-II. The existence of scenario (iii) is justified by Angenent, Altschuler and Giga \cite{AAG}. The quantitative precise asymptotics for such Type-II solutions have been obtained by Angenent and Vel\'{a}zquez in \cite{AV97}.

For each integer $m\geqslant 3$, Angenent and Vel\'{a}zquez \cite{AV97} construct a (mean convex) rotationally symmetric MCF on an $n$-sphere ($n\geqslant 2$) (centred at the origin) shrinking to a point (the origin) in a ``non-convex'' fashion in finite time $T$. If $m$ is even, the solution has reflexive symmetry across the equator and corresponds to the aforementioned scenario (iii); if $m$ is odd, the solution looks like an asymmetric ``dumbbell'' and we refer the reader to \cite{AV97}.

The geometric-analytic features of an Angenent-Vel\'{a}zquez solution can be summarized as follows: (1) At each pole (the ``tip'') of the sphere (an umbilical point), the curvature blows up at the Type-II rate $(T-t)^{-(1-1/m)}$ and the singularity model there is the bowl soliton, which is the unique (up to rigid motion) translating soliton that is rotationally symmetric and strictly convex \cite{Has15}. (2) Near the equator (the ``neck''), the curvature blows up at the Type-I rate and the singularity model there is the shrinking soliton. (3) Between each pole and the equator, the solution is approximately given by rotating the profile (with the $x$-axis being the axis of rotation) $u^2+Kx^m = 2(n-1)(T-t)$, $t\in[t_0,T)$, for some positive constant $K$. These examples are all believed to be ``rare'', as is reflected by the fact that their Type-II curvature blow-up rates are discrete and quantized. Indeed, by the fundamental work of Colding and Minicozzi \cite{CM12}, we know these solutions are non-generic. We note that, integrating $(T-t)^{-(1-1/m)}$ in $t$, the tip moves by a finite distance over the time interval $[t_0,T)$.

Having established the existence of compact MCF solutions with Type-II curvature blow-up, it is natural to seek \emph{noncompact} counterparts, as first realized in \cite{IW19} by the first two authors of this paper. More precisely, for each real number $\gamma>1/2$, we have constructed MCF of noncompact rotationally symmetric embedded hypersurfaces that are complete convex graphs over a shrinking ball and asymptotically approach a shrinking cylinder near spatial infinity. Such a mean curvature flow solution exhibits the following behaviour near the ``vanishing'' time $T$: (1) The highest curvature, concentrated at the tip of the hypersurface (an umbilical point), blows up at the rate $(T-t)^{-(\gamma +1/2)}$ where $\gamma>1/2$, and the singularity model there is the bowl soliton. (2) Near spatial infinity, the hypersurface approaches a collapsing cylinder at a power decay rate dependent on the parameter $\gamma$. (3) Between the tip and the cylindrical end, the solution is approximately given by rotating the profile (with the $x$-axis being the axis of rotation) $u^2+Kx^{\frac{1}{1/2-\gamma}} = 2(n-1)(T-t)$, $t\in[t_0,T)$, for some positive constant\footnote{In this paper, constants $K$ and $C$ may change from line to line.} $K$ .

The Isenberg-Wu solutions and the Angenent-Vel\'{a}zquez solutions share similar geometric features---in particular, in both cases, the solutions join a translating soliton to a shrinking soliton. Yet they are different in terms of the topology of the hypersurfaces and the geometric-analytic features. In particular, the noncompact Isenberg-Wu solutions seem to be much more ``abundant'' than the compact Angenent-Vel\'{a}zquez ones, as is reflected by the fact that their Type-II curvature blow-up rates form a continuum $(1,\infty)$ and we have an open set of solutions for each $\gamma>1/2$. We note that, integrating $(T-t)^{-(\gamma +1/2)}$ in $t$ and because $\gamma>1/2$, the tip moves by an infinite distance over the finite time interval $[t_0,T)$, so the MCF solution disappears at spatial infinity at $T$, exactly when the asymptotic cylinder collapses to a line. The general behaviour, but not the precise asymptotics of such solutions is studied in \cite{SS14}.

An inspection of the Angenent-Vel\'{a}zquez solutions and the Isenberg-Wu solutions immediately raises the following question: does there exist a Type-II MCF solution with curvature blow-up rate $(T-t)^{-1}$? The existence is suggested by taking the appropriate limit of the parameter in either construction: 
\begin{align*}
\lim\limits_{m\to\infty}{(T-t)^{-1+\frac{1}{m}}} = \lim\limits_{\gamma\to\frac{1}{2}+} (T-t)^{-(\gamma+1/2)} = (T-t)^{-1}.
\end{align*}
Further motivation comes from the differences which have been observed between the Type-II solutions in \emph{Ricci flow} on compact manifolds and those seen on noncompact manifolds. On compact manifolds $\Sigma$, all the examples that  have been found \cite{AIK15} have ``quantized''  blowup rates of $\sup\limits_{x\in \Sigma}|\Rm(x,t)| \sim  (T-t)^{\frac{2}{k}-2}$ for integers $k\geqslant 3$ (here $T$ is the time of the first singularity). By contrast, for noncompact manifolds $\Sigma$, the known examples \cite{Wu14} have a continuous spectrum of blowup rates:  $\sup\limits_{x\in \Sigma}|\Rm(x,t)| \sim (T-t)^{-\lambda-1}$ for all $\lambda \geqslant 1$. The borderline Type-II rate $(T-t)^{-2}$ (letting $\lambda=1$ or $k\to\infty$) in Ricci flow can be realized on a noncompact manifold. Correspondingly, the borderline Type-II rate $(T-t)^{-1}$ in MCF can be expected on a noncompact hypersurface. In this paper, we confirm this expectation.

Following the set up in \cite{IW19}, in this paper we consider mean curvature flow of rotationally symmetric hypersurfaces embedded in Euclidean space. For any point $(x_0,x_1,\ldots,x_n)\in\mathbb{R}^{n+1}$ for $n\geqslant 2$, we write
\begin{align*}
x = x_0, \quad r = \sqrt{x_1^2 + \cdots + x_n^2}.
\end{align*}
A noncompact hypersurface $\Gamma$ is said to be rotationally symmetric if
\begin{align*}
\Gamma = \left\{(x_0,x_1,\ldots,x_n) : r = u(x), a\leqslant x < \infty \right\}.
\end{align*}
The rotational symmetry is preserved along MCF, for example, by the obvious curvature bound in our consideration and the standard uniqueness result.

We assume that $u$ is strictly concave so that the hypersurface $\Gamma$ is convex and that $u$ is strictly increasing with $u(a)=0$ and with $\lim\limits_{x\nearrow \infty}u(x) = r_0$, where $r_0$ is the radius of the cylinder. The function $u$ is assumed to be smooth, except at $x=a$. Note that this particular non-smoothness of $u$ is a consequence of the choice of the (cylindrical-type) coordinates; in fact, as seen below, if the time-dependent flow function $u(x,t)$ is inverted in a particular way, this irregularity is removed. We label the point where $u=0$ the \emph{tip} of the surface.

We focus our attention on the class of complete hypersurfaces that are rotationally symmetric, (strictly) convex\footnote{Throughout this paper, ``convex'' means ``strictly convex''.}, smooth graphs over a ball and asymptotic to a cylinder. One readily verifies that embeddings with these properties are preserved by MCF (see for example \cite{SS14}). Representing the evolving hypersurface $\Gamma_t$ by the profile of rotation, i.e., the graph of $r=u(x,t)$, then the function $u$ satisfies the PDE,
\begin{align}
\label{eq:u(x,t)}
u_t & = \frac{u_{xx}}{1+u^2_x} - \frac{n-1}{u}.
\end{align}
We introduce the following scaled time and space parameters, and the scaled profile function:
\begin{align*}
\tau  &= -\log(T-t),\\
\quad y & = x + a \log(T-t),\\
\quad \phi(y,\tau) & = u(x,t)(T-t)^{-1/2},
\end{align*}
where $a>0$ is to be chosen later.

Under the rescaled parameters, \eqref{eq:u(x,t)} for $u(x, t)$ is transformed to the following PDE for $\phi(y,\tau)$:
\begin{align}\label{eq:phi(y,tau)}
\left. \p_\tau \right\vert_y \phi & = \frac{e^{-\tau} \phi_{yy}}{1+e^{-\tau}\phi^2_y} + a \phi_y
+ \frac{\phi}{2} - \frac{(n-1)}{\phi} ,
\end{align}
where $\left. \p_\tau \right\vert_y $ means taking the partial derivative for the variable $\tau$ with respect to the coordinates $(y, \tau)$; in other words, with $y$ fixed. This notation appears repeatedly throughout this paper. We readily note that equation \eqref{eq:phi(y,tau)} admits the constant solution $\phi \equiv \sqrt{2(n-1)}$, which  corresponds to the collapsing cylinder (a shrinking soliton).

Because our hypersurface is assumed to be a complete convex graph over a ball, it is useful to invert the coordinates and work with 
\begin{align*}
y( \phi,\tau) & = y\left( \phi(y,\tau), \tau \right).
\end{align*}
This inversion can be done because the hypersurface under consideration is a convex graph over a ball. In terms of $y(\phi,\tau)$, the equation corresponding to mean curvature flow (equivalent to equation \eqref{eq:phi(y,tau)} and also \eqref{eq:u(x,t)}) is
\begin{align}
\label{eq:y(phi,tau)}
\left. \p_\tau \right\vert_{\phi} y & = \frac{y_{\phi\phi}}{1 + e^{\tau} y^2_\phi} + \left(\frac{(n-1)}{\phi} - \frac{\phi}{2} \right) y_\phi - a.
\end{align}

Our main result is the following. 
\begin{thm}
\label{thmmain}
For any choice of an integer $n\geqslant 2$ and for any real number $a>0$, there exists a family $\mathscr{G}$ of $n$-dimensional, smooth, complete noncompact, rotationally symmetric, strictly convex hypersurfaces in $\mathbb{R}^{n+1}$ such that the MCF evolution $\Gamma_t$ starting at each hypersurface $\Gamma\in\mathscr{G}$ is trapped in a shrinking cylinder,  escapes at spatial infinity while the cylinder becomes singular at $T<\infty$, and has the following precise asymptotic properties near the vanishing time $T$ of $\Gamma_t$:
\begin{enumerate}
\item The highest curvature occurs at the tip of the hypersurface $\Gamma_t$ , and it blows up at the precise Type-II rate
\begin{align}
\sup\limits_{p\in M_t}|h(p,t)| & \sim (T-t)^{-1} \quad \text{as } t\nearrow T.
\end{align}
\item Near the tip, the Type-II blow-up of $\Gamma_t$ converges to a translating soliton which is a higher-dimensional analogue of the ``Grim Reaper''
\begin{align}
y\left(e^{-\tau/2} z,\tau\right)=y(0,\tau) + e^{-\tau} \left(\frac{\tilde P\left( a z \right)}{a}+o(1)\right)\quad \text{as } \tau\nearrow \infty
\end{align}
uniformly on compact $z$ intervals, where $z=\phi e^{\tau/2}$ and $\tilde P$ is defined in equation \eqref{eq:tildeP}.
\item Away from the tip and near spatial infinity, the Type-I blow-up of $\Gamma_t$  approaches the cylinder at the rate 
\begin{align}
2(n-1)-\phi^2 & \sim e^{-y/a} \quad \text{as } y\nearrow \infty.
\end{align}
\end{enumerate}
In particular, the solution\footnote{Any such solution is unique by work of Chen and Yin \cite{CY07}} constructed has the asymptotics predicted by the formal solution described in Section \ref{formal}.
\end{thm}

Comparing with \cite{IW19}, we see that the definitions of $\tau$ and $\phi$ remain the same but that of $y$ has changed from $y=x(T-t)^{\gamma-1/2}$ in \cite{IW19} to $y= x+a\log(T-t)$ in the present paper. Indeed, to capture the borderline case $\gamma=1/2$, taking the limit $\gamma\to\frac{1}{2}$ in \cite{IW19} is insufficient. The new scaling of $y$, on the other hand, is natural because in \cite{IW19} the asymptotic cylindricality is measured precisely by $2(n-1)-\phi^2\sim y^{(1/2-\gamma)^{-1}}$ and if we let $\gamma\to 1/2$, then we expect $2(n-1)-\phi^2$ to decay faster than any arbitrarily large power of $y$; i.e., we have exponential decay in $y$, as is captured by the asymptotic property (3) of Theorem \ref{thmmain}. The new scaling of $y$ implies changes in the rescaled PDEs for MCF; e.g., equations \eqref{eq:phi(y,tau)} and \eqref{eq:y(phi,tau)}, cf. the same-numbered equations in \cite{IW19}. In particular, we note that the change occurs in the first-order term in equation \eqref{eq:phi(y,tau)}, or equivalently in the zeroth-order term in equation \eqref{eq:phi(y,tau)}. This suggests that the method of construction in \cite{IW19} is still applicable.

The proof of this theorem is based on matched asymptotic analysis and barrier arguments for nonlinear PDE. While the analysis is intricate, this method is powerful and has been successfully applied in a number of studies of Type-I and Type-II singularities which develop both in Ricci flow  \cite{AK07, AIK15, Wu14} and in MCF \cite{AV97, IW19}. The proof proceeds in the following steps: (1) By considering rotationally symmetric hypersurfaces, we reduce the MCF equation to a quasilinear  parabolic PDE for a scalar function. (2) Applying matched asymptotic analysis, we formally construct approximate solutions to the rescaled versions of this PDE. (3) For each such approximate solution, we construct subsolutions and supersolutions which, if carefully patched, form barriers for the rescaled PDE. These barriers carry information of the approximate solution for times very close to the vanishing time $T$. (4) Once we have shown (using a comparison principle) that any solution starting from initial data between the barriers does stay between them up to time $T$, and once we have determined that such initial data sets do exist, we can conclude that there are MCF solutions whose behaviours are predicted by the barriers. Near spatial infinity, the barriers give us precise measure of the asymptotic cylindricality of a solution. At the tip, the barriers have Type-II speeds (cf. Section \ref{formal}), which implies that on average any MCF solution in between is also Type-II. However, this alone does not imply the stronger convergence result of Type-II blow-up. To prove the strong convergence result stated in property (2) of the theorem, we rely on Lemma \ref{lem:type2}.

This paper is organized as follows. Section \ref{formal} describes the construction of formal solutions using the method of formal matched asymptotics. In Section \ref{super-sub}, we use these formal solutions to construct the corresponding supersolutions and subsolutions to the rescaled PDE. The supersolutions and subsolutions are ordered and patched to create the barriers to the rescaled PDE in Section \ref{barriers}; a comparison principle for the subsolutions and supersolutions is also proved there. In Section \ref{existence}, we use these results to complete the proof of our main theorem. 

\subsection*{Acknowledgements}
We thank Dan Knopf for helpful discussion on this project. J. Isenberg is partially supported by NSF grant PHY-1707427; H. Wu thanks the support by ARC grant DE180101348; Z. Zhang thanks the support by ARC grant FT150100341.

\section{Formal solutions}\label{formal}

To begin the derivation of a class of formal approximate solutions, we assume that for large values of $\tau$, the terms $\left. \p_\tau \right\vert_{\phi} y$ and $\displaystyle{\frac{y_{\phi\phi}}{1+e^{\tau}y^2_\phi}}$ in equation \eqref{eq:y(phi,tau)} are negligible. It follows that the PDE \eqref{eq:y(phi,tau)} can be approximated by the ODE
\begin{align}
\label{truncd}
\left(\frac{(n-1)}{\phi}-\frac{\phi}{2}\right) \tilde y_\phi - a = 0, 
\end{align}
for which the general solution takes the form
\begin{align}
\label{eq:truncdsoln}
\tilde y (\phi) & = C_1 - a \log\left(2(n-1)-\phi^2\right),
\end{align}
where $C_1$ is an arbitrary constant, and $\phi\in[0,\sqrt{2(n-1)})$. Note that $y(\phi)$ is convex and $\tilde y \nearrow \infty$ as $\phi\nearrow\sqrt{2(n-1)}$. This is consistent with the hypersurface being asymptotic to a cylinder at spatial infinity, which is a desired feature for the solutions of interest.

In light of the assumptions made at the beginning of this section in obtaining the ODE \eqref{truncd}, we substitute the solution $\tilde y$ into the quantity $\displaystyle{\frac{y_{\phi\phi}}{1+e^{\tau}y^2_\phi}}$, obtaining 
\begin{align*}
\frac{\tilde y_{\phi\phi}}{1 + e^{\tau} {\tilde y_\phi}^2} & = \frac{2a(2n-2+\phi^2)}{(2n-2-\phi^2)^2 + 4a^2 e^\tau \phi^2 }.
\end{align*}
This suggests that  $\tilde y$ is a reasonable approximate solution, provided that $\phi e^{\tau/2}$ is sufficiently large. 

We now set $z:=\phi e^{\tau/2}$ and label the dynamic (i.e. time-dependent) region in which $z=O(1)$ as the \emph{interior region}. The complement of the interior region is labelled as the \emph{exterior region}. 

Note that the condition $z=\phi e^{\tau/2} =O(1)$ is equivalent to the condition $\phi = O\left( e^{-\tau/2}\right)$, which corresponds to a region near the tip (at which $\phi = 0$). Using the change-of-variables formula
\begin{align*}
\p_\tau|_z y = \p_\tau|_{\phi} y - \frac{1}{2} z y_z,
\end{align*}
and \eqref{eq:y(phi,tau)}, we obtain the evolution equation for $y(z,\tau)$:
\begin{align}\label{eq:y(z,tau)}
\left. \p_\tau\right\vert_z y & = \frac{y_{zz}}{e^{-\tau} + e^{\tau}y^2_z} +  \left(e^{\tau}\frac{(n-1)}{z} - z \right) y_z - a.
\end{align}

As in \cite{IW19}, we consider the ansatz
\begin{align}
\label{eq:tipansatz-y}
y = \tilde A + e^{-\tau} \tilde F(z,\tau),
\end{align}
where $\tilde A$ is a constant. Substituting \eqref{eq:tipansatz-y} into equation \eqref{eq:y(z,tau)}, we obtain 
\begin{align}
\label{eq:tildeAF}
\frac{\tilde F_{zz}}{1 + \tilde F^2_z} + (n-1) \frac{\tilde F_z}{z} & = a + e^{-\tau}\left( z \tilde F_z - \tilde F + \left.\p_\tau\right\vert_z \tilde F \right). 
\end{align}
Continuing the formal argument, we assume that for $\tau$ very large, the term in \eqref{eq:tildeAF} with the coefficient $e^{-\tau}$ is negligible. Then equation \eqref{eq:tildeAF} is reduced to the ODE
\begin{align}
\label{eq:tildeF}
\frac{\tilde F_{zz}}{1+\tilde F^2_z} + (n-1)\frac{\tilde F_z}{z} & = a
\end{align}
for some constant $a$. To solve \eqref{eq:tildeF} for $\tilde F$, we define $\tilde P(w)$ to be the unique solution to the initial value problem  
\begin{align}
\label{eq:tildeP}
\frac{\tilde P_{ww}}{1+(\tilde P_w)^2} + (n-1)\frac{\tilde P_w}{w} = 1,\quad  \tilde P(0) = \tilde P_w(0) = 0,
\end{align}
which is clearly an even function by symmetry and uniqueness. Note that $w=0$ is a regular singular point. We then readily verify that for an arbitrary function $C(\tau)$, 
\begin{align}
\label{eq:tildeF(z,tau)}
\tilde F(z,\tau) = \frac{1}{a} \tilde P ( a z ) + C(\tau)
\end{align}
satisfies \eqref{eq:tildeF}. 

\begin{remark}

For the dimension $n=1$, \eqref{eq:tildeP} is reduced to
\begin{align*}
\frac{\tilde P''}{1+(\tilde P')^2} = 1,\quad  \tilde P(0) = \tilde P'(0) = 0,
\end{align*}
whose solution is $\tilde P(r) = -\log \cos r$. The graph of $x= t-\log \cos r$, where $r\in(-\pi/2,\pi/2)$, has been named the ``Grim Reaper'' by M. Grayson \cite{Ang91}. It translates with constant velocity along the $x$-axis and is a solution to the curve-shortening flow (i.e., 1-dimensional MCF). 

For $n\geqslant 2$, solving \eqref{eq:tildeP} for the function $\tilde P$ and then rotating the graph of $x=c^{-1}\tilde{P}(cr)+ct$ around the $x$-axis defines a higher dimensional analogue of the Grim Reaper, a translating soliton also known as the bowl soliton. 
\end{remark}

The initial value problem \eqref{eq:tildeP} has been solved in \cite[pp.24--25]{AV97} for general dimensions, which has a unique convex solution with the following asymptotics: 
\begin{align}\label{eq:asymptotics-tildeP}
\tilde P(z) &= \left\{
\begin{array}{cc}
\frac{z^2}{2n} + o\left(z^2\right) & z\searrow 0\\ \\
\frac{z^2}{2(n-1)} - \log z + O\left(z^{-2}\right)& z\nearrow\infty.
\end{array}
\right.
\end{align}
where the $z\nearrow\infty$ case is derived in \cite[Proposition 2.1]{AV97} and the $z\searrow 0$ case is obvious by the equation and demonstrates the smoothness through $z=0$. It then follows that the asymptotics for $y(z,\tau)$ take the form
\begin{align}\label{eq:asymptotics-y}
y(z) &= \left\{
\begin{array}{cc}
\tilde A + e^{-\tau} C(\tau) + e^{-\tau} \frac{a}{2n} z^2+ o\left(e^{-\tau}z^2\right) & z\searrow 0\\ \\
\vspace{10pt}
\tilde A + e^{-\tau} \left(C(\tau)-\frac{1}{a}\log a \right)+ e^{-\tau} \frac{a}{2(n-1)} z^2-e^{-\tau}\frac{1}{a}\log z+ O\left(e^{-\tau}z^{-2}\right) & z\nearrow\infty.
\end{array}
\right.
\end{align}

We now discuss the properties of the formal solutions. Recalling the scaling formulas $x = y - a\log(T-t)$ and $z = u (T-t)^{-1}$, as well as the interior region ansatz formula  \eqref{eq:tipansatz-y} and the expression \eqref{eq:asymptotics-tildeP} for the asymptotics of $\tilde P$,  we obtain the following asymptotic expression for $x$ in a neighbourhood of the tip (i.e. for $z$ near $0$):
\begin{align*}
x & = y - a \log(T-t) \\
& =  \tilde A + (T-t) C\left(-\log(T-t)\right) - a\log(T-t) +\frac{a}{2n} \frac{u^2}{(T-t)} + o\left(\frac{u^2}{T-t}\right).
\end{align*}

In our consideration, as $t\nearrow T$, the highest curvature always occurs at the tip, which is an umbilical point (cf. item (1) of Theorem \ref{thmmain}), so the mean curvature and hence the normal (horizontal) velocity attain their maximal values there. Using the asymptotic expression for $x$ from above, we have 
\begin{align}\label{eq:meancurv}
\left. H \right\vert_{\text{tip}} & = n \left. \frac{d^2 x}{du^2}\right   \vert_{u=0} = \frac{a}{T-t}
\end{align}
which implies that the curvature at the tip blows up at Type-II rate. Moreover, over the time period $[t_0,T)$, the tip of the surface moves  along the $x$-axis to the right from its initial position $x_0$ by the amount
\begin{align*}
\int_{t_0}^{T} H \; ds & = \lim\limits_{t\nearrow T} \int_{t_0}^{t} \left. H\right\vert_{\text{tip}} \; ds\\
& = \lim\limits_{t\nearrow T} -a\log(T-t) +a\log(T-t_0) \\
& = \infty.
\end{align*}
Hence, we see that in terms of the original $x$-coordinate, the surface evolving by MCF disappears off to spatial infinity as $t\nearrow T$. However in terms of  the $y$-coordinate, if for example, we choose $C(\tau)=O(\tau)$ (cf. Section 3.1), then the tip remains a finite distance from the origin for all time $\tau$ since
\begin{align*}
y_0(\tau) & = \tilde A + e^{-\tau}C(\tau)\\
& \approx \tilde A.
\end{align*}

The formal solutions constructed separately in the interior and the exterior regions each involve a free parameter. Matching the formal solutions on the overlap of the two regions, we can establish an algebraic relationship between them. 

In the interior region, for the large $z$ asymptotic expansion formula \eqref{eq:asymptotics-y} for the solution $y(z)$, by setting $z$ equal to a large constant $R$ and presuming that $\tau$ is very large, one has 
\begin{align}\label{eq:match-formal-int}
y & =    \tilde A + e^{-\tau} \left(C(\tau)-\frac{1}{a}\log a \right)+ e^{-\tau} \frac{a}{2(n-1)} z^2-e^{-\tau}\frac{1}{a}\log z+ O\left(e^{-\tau}z^{-2}\right) \notag\\
& \approx \tilde A.
\end{align}

In the exterior region, also setting $z=R$ (and so $\phi = Re^{-\tau/2}$) and presuming very large $\tau$, we have from (\ref{eq:truncdsoln})
\begin{align}
\label{eq:match-formal-ext}
y 
& = \tilde y (\phi) \notag \\
& = C_1 - a \log\left(2(n-1)-\phi^2\right) \notag \\
& \approx C_1 - a\log(2n-2).
\end{align}
Matching \eqref{eq:match-formal-int} with \eqref{eq:match-formal-ext}, we obtain 
\begin{align*}
\tilde A = C_1 - a\log(2n-2).
\end{align*}

We now collect these results and write out expressions for our formal solutions, both in the interior and the exterior regions. In the interior region, which is characterised by $z = \phi e^{\tau/2} = u e^{\tau} = O(1)$, we blow up the  MCF solution $u(t,x)$  at the prescribed Type-II rate $(T-t)^{-1}$. We also rescale the coordinates in accord with how fast the surface moves under mean curvature flow. Then in the interior region, the formal solution is given by 
\begin{align*}
y_{form, int} & = \tilde A +  e^{-\tau} C(\tau) + e^{-\tau} \tilde F(z),
\end{align*}
where $\tilde F$ and the as-yet-unspecified function $ C(\tau)$ are related to $\tilde P$ as in \eqref{eq:tildeF(z,tau)}, and where $\tilde P$ is the solution to the initial value problem \eqref{eq:tildeP}. 

In the exterior region, where $R e^{-\gamma\tau} \leqslant \phi < \sqrt{2(n-1)}$ for some large $R>0$, the formal solution takes the form
\begin{align*}
y_{form, ext} & = \tilde A + a\log(2n-2) - a\log(2n-2-\phi^2).
\end{align*}
We note that $y\nearrow\infty$ as $\phi\nearrow\sqrt{2(n-1)}$, which indicates that the exterior formal solutions are  asymptotic to and lie strictly within the cylinder of radius $\sqrt{2(n-1)}$.

\subsection{The formal solutions revisited in the form  $\lambda(z,\tau)$ or $\lambda(\phi, \tau)$}\label{lambdaztau}

To prove the main result, Theorem \ref{thmmain}, it is useful to work with the quantity $\lambda := -1/y$, since in terms of $\lambda$, the asymptotically cylindrical end of the embedded hypersurface corresponding to large values of $y$ is effectively compactified. The MCF evolution equation for $\lambda$ is readily obtained by substituting $\lambda=-1/y$ into \eqref{eq:y(phi,tau)}:
\begin{align}\label{eq:lambda(phi,tau)}
\left. \p_\tau \right\vert_\phi \lambda & = \frac{\lambda_{\phi\phi}-2\lambda^2_\phi/\lambda}{1+e^{\tau}\lambda^2_\phi/\lambda^4} + \left(\frac{n-1}{\phi} - \frac{\phi}{2}\right)\lambda_\phi - a \lambda^2.
\end{align}
The class of MCF solutions we consider here corresponds to solutions of equation \eqref{eq:lambda(phi,tau)} subject to the following effective boundary conditions: the rotational symmetry of the graph implies that $\lambda_\phi (0,\tau) = 0$, and the asymptotically cylindrical condition implies that $\lambda(\sqrt{2(n-1)}, \tau)=\lambda(-\sqrt{2(n-1)}, \tau)=0$.

As in the analysis done above in terms of $y$, it is useful here to use the dilated spatial variable $z=\phi e^{\tau/2}$. The evolution equation for $\lambda(z,\tau)$ takes the form
\begin{align}\label{eq:lambda(z,tau)}
\left. \p_\tau \right\vert_z \lambda & = \frac{e^{\tau}(\lambda_{zz}-2\lambda^2_z/\lambda)}{1+e^{2\tau}\lambda^2_z/\lambda^4} + e^{\tau}(n-1)\frac{\lambda_z}{z} - z\lambda_z - a \lambda^2.
\end{align}

We now construct the formal solutions in terms of $\lambda(z,\tau)$ or $\lambda(\phi,\tau)$, using arguments very similar to those above in terms of $y$. 

In the interior region, where $z=O(1)$, we use the ansatz
\begin{align*}
\lambda=-A+e^{-\tau}F(z),
\end{align*}
where $A$ is a positive constant. Substituting this ansatz into equation \eqref{eq:lambda(z,tau)}, we find $F$ satisfying
\begin{align}
\label{Flambda}
e^{-\tau}\left(-F + \left. \partial_\tau\right\vert_z F \right)
& = \frac{F_{zz} - 2e^{-\tau} F^2_z/(-A+e^{-\tau}F)}{1+F^2_z/(-A+e^{-\tau} F)^4} + (n-1)\frac{F_z}{z} - zF_ze^{-\tau}\\
\nonumber &\quad  - a A^2 + 2aAFe^{-\tau} - aF^2e^{-2\tau}.
\end{align}
Assuming, in the formal argument, that the terms with the coefficient $e^{-\tau}$ in equation \eqref{Flambda} can be ignored for large $\tau$, \eqref{Flambda} is reduced to the following ODE for $F$:
\begin{align}
\label{eq:ODE-F}
\frac{F_{zz}}{1+F^2_z/A^4} + (n-1) \frac{F_z}{z} = aA^2.
\end{align}
To solve \eqref{Flambda}, we rescale $F$ according to
\begin{equation}
\label{FP}
F(z)= \frac{A^2}{a} P(az),
\end{equation}
and determine that $P(w)$ satisfies the ODE for $w=az$, 
\begin{align*}
\frac{P_{ww}}{1 + P_{w}^2} + (n-1)\frac{P_w}{w} = 1.
\end{align*}
Subject to the initial conditions $P(0)=P_w(0)=0$ which naturally come from the geometric interpretation of $\lambda$, we can solve for $P$ uniquely; cf., equation \eqref{eq:tildeP}. Moreover, the asymptotic expansions of $P(w)$ are known:
\begin{align*}
P(w) &= \left\{
\begin{array}{cc}
\frac{1}{2n} w^2 + o\left(w^2\right) & w\searrow 0 \\ \\
\frac{1}{2(n-1)} w^2 - \log w + O\left(w^{-2} \right) & w\nearrow\infty
\end{array}
\right.
\end{align*}
where the $w\nearrow\infty$ case is derived in \cite[Proposition 2.1]{AV97} and the $w\searrow 0$ is obvious by the equation. Consequently, the asymptotic expansions of $F(z)$ are as follows:
\begin{align}\label{eq:asymp-F}
F(z) &= \left\{
\begin{array}{cc}
\frac{aA^2}{2n} z^2 + o\left(z^2\right) & z\searrow 0 \\ \\
\frac{aA^2}{2(n-1)} z^2 - \frac{A^2}{a}\log(az) + O\left( z^{-2}\right) & z\nearrow\infty.
\end{array}
\right.
\end{align}

In the exterior region, examining  the evolution of $\lambda(\phi,\tau)$ as governed by \eqref{eq:lambda(phi,tau)}, we assume (in the formal argument) that the term $\displaystyle{\frac{\lambda_{\phi\phi}-2\lambda^2_\phi/\lambda}{1+e^{\tau}\lambda^2_\phi/\lambda^4}}$ is negligible for $\tau$ large. Then in the same way as $\tilde y$ in \eqref{eq:truncdsoln} for $y$, we note that any solution of the limiting equation
\begin{align}
\label{eq:ODE-lambdabar}
\left(\frac{n-1}{\phi} - \frac{\phi}{2} \right)\bar\lambda_\phi - a \bar\lambda^2 = 0
\end{align}
is an approximate solution to \eqref{eq:lambda(phi,tau)}. We can solve for $\bar\lambda(\phi)$ explicitly, 
\begin{align}
\label{eq:lambdabar}
\bar\lambda(\phi) & = - \frac{1}{C_1 - a \log(2n-2-\phi^2)}
\end{align}
for any constant $C_1>a\log(2n-2)$.

\section{Supersolutions and subsolutions}\label{super-sub}

For a differential equation of the form $\mathcal{D}[\psi]=0$, a function $\psi_+$ is a \emph{supersolution} if $\mathcal{D}[\psi^+]\geqslant 0$, while $\psi^-$ is a \emph{subsolution} if $\mathcal{D}[\psi^-]\leqslant 0$. If there exist a supersolution $\psi^+$ and a subsolution $\psi^-$ for the differential operator $\mathcal{D}$, and they satisfy the inequality $\psi^+ \geqslant \psi^-$, then they are called \emph{upper and lower barriers}, respectively. If $\mathcal{D}[\psi]=0$ admits solutions, then the existence of barriers $\psi^+\geqslant \psi^-$ implies that there exists a solution $\psi$ with $\psi^+ \geqslant \psi \geqslant \psi^-$. This is the general idea of our argument which we justify during the procedure. 

In this section, we construct subsolutions and supersolutions for the MCF of our models in the interior and the exterior regions. Then in the next section, we combine them to obtain the global barriers for the flow.

\subsection{Interior region}

In the interior region, we work with  $\lambda(z,\tau)$ and the corresponding  MCF equation \eqref{eq:lambda(z,tau)}. The differential operator is the following quasilinear parabolic one 
\begin{align}
\label{Tz}
\mathcal{T}_z[\lambda] : = & \left. \p_\tau \right\vert_z \lambda - \frac{e^{\tau}(\lambda_{zz}-2\lambda^2_z/\lambda)}{1+e^{2\tau}\lambda^2_z/\lambda^4} - e^{\tau}(n-1)\frac{\lambda_z}{z} + z\lambda_z + a \lambda^2,
\end{align}
for which we seek subsolutions and supersolutions. The result is the following. 

\begin{lemma}\label{interior-supersub}
For an integer $n\geqslant 2$, a constant $a>0$ and a pair of positive numbers $A^{\pm}$, we define even functions $F^{\pm}$ to be the solution to equation \eqref{eq:ODE-F} with $A=A^{\pm}$ respectively.

For any fixed constants $R_1>0$, $B^{\pm}$ and $E^{\pm}$, there exist even functions $Q^{\pm}:\RR\to\RR$, constants $D^{\pm}$, and a sufficiently large $\tau_1<\infty$ such that the functions
\begin{align}
\label{eq:int-supersub}
\lambda^{\pm}_{int} (z,\tau) := -A^{\pm} + e^{-\tau}F^{\pm}(z) + e^{-\tau} \left(B^{\pm}\tau + E^{\pm}\right)  + \tau e^{-2\tau}D^{\pm}Q^{\pm}(z)
\end{align}
are a supersolution ($+$) and a subsolution ($-$) respectively of $\mathcal{T}_z[\lambda]=0$ on the interval $|z| \leqslant R_1$ for all $\tau\geq\tau_1$.

The functions $Q^{\pm}$ depend on $A^{\pm}$ and $F^{\pm}(z)$ respectively. The constants $D^{\pm}$ depend on $n$, $R_1$, $A^{\pm}$ and $B^{\pm}$ respectively.
\end{lemma}

\begin{proof}

In the proof, we omit the $\pm$ in the notations as the argument is the same for both. The difference only shows up at the end of the argument, as is specified below. The functions involved are all even in $z$, so we focus on $z\geqslant 0$. 

We start our proof with the function $Q$ unspecified; $Q$ is to be determined in \eqref{eq:ODE-Q}. Applying the operator $\mathcal{T}_z$ (defined in \eqref{Tz}) to the function $\lambda ^+_{int}$ from \eqref{eq:int-supersub}, we calculate
\begin{align*}
\mathcal{T}_z[\lambda^+_{int}] & = I + II + III + IV + V,
\end{align*}
where (to simplify the expressions, we replace ``$\lambda^+_{int}$'' by ``$\lambda$'')
\begin{align*}
I & = \left. \p_\tau \right\vert_\phi \lambda\\
& = - B \tau e^{-\tau} -e^{-\tau}F  + B e^{-\tau} - E e^{-\tau} + O\left(\tau e^{-2\tau}\right),\\
III & = - e^{\tau}(n-1) \frac{\lambda_z}{z}\\
& = -(n-1)\frac{F_z}{z} - (n-1) \frac{Q_z}{z} D\tau e^{-\tau},\\
IV & = z\lambda_z \\
& = z F_z e^{-\tau} + \tau e^{-2\tau} D z Q_z,\\
V & = a\lambda^2 \\
& = aA^2 - 2aA\left( B\tau e^{-\tau} + (E+F)e^{-\tau} \right) + O\left(\tau^2e^{-2\tau}\right),
\end{align*}
and the most complicated term
$$II= - \frac{e^{\tau}(\lambda_{zz}-2\lambda^2_z/\lambda)}{1+e^{2\tau}\lambda^2_z/\lambda^4}.$$
Now we calculate the expansion of $II$ with respect to $\tau$. Recall that $F$ and $Q$ are even functions of $z$ and the other capital letters in these expressions are constants. We have 
\begin{align*}
\lambda & =-A+\tau e^{-\tau}B+e^{-\tau}(E+F)+\tau e^{-2\tau}DQ,\\
\lambda_z & =e^{-\tau}F_z+\tau e^{-2\tau}DQ_z,\\
\lambda_{zz} & =e^{-\tau}F_{zz}+\tau e^{-2\tau}DQ_{zz}.
\end{align*}
If we substitute these quantities into the expression for $II$, we obtain
\begin{equation}
II= -\frac{\lambda^4(F_{zz}+\tau e^{-\tau}DQ_{zz})-2\lambda^3e^{-\tau}(F_z^2+2\tau e^{-\tau}DF_zQ_z+\tau^2 e^{-2\tau}D^2Q_z^2)}{\lambda^4+F_z^2+2\tau e^{-\tau}DF_zQ_z+\tau^2 e^{-2\tau}D^2Q_z^2} \nonumber
\end{equation}
For our purposes, we only need to keep track of the first two leading-order terms, i.e., the coefficient of $1$ and $\tau e^{-\tau}$, and so the quantity $II$ takes the form
\begin{equation}
\begin{split}
II
&= -\frac{(A^4-4A^3B\tau e^{-\tau})(F_{zz}+\tau e^{-\tau}DQ_{zz})}{A^4+F_z^2+(2DF_zQ_z-4A^3B)\tau e^{-\tau}}+O(e^{-\tau}) \\
&= -\frac{(A^4-4A^3B\tau e^{-\tau})(F_{zz}+\tau e^{-\tau}DQ_{zz})}{A^4+F_z^2} \left(1-\tau e^{-\tau}\frac{2DF_zQ_z-4A^3B}{A^4+F_z^2}\right)+O(e^{-\tau}) \\
&= -\frac{A^4F_{zz}}{A^4+F_z^2}-\frac{\tau e^{-\tau}}{A^4+F_z^2}\left(-A^4F_{zz}\frac{2DF_zQ_z-4A^3B}{A^4+F_z^2}+A^4DQ_{zz}-4A^3BF_{zz}\right)+O(e^{-\tau}), \nonumber
\end{split}
\end{equation}
where the $O(e^{-\tau})$ terms are uniform with respect to $|z|\leqslant R_1$. 

Thus for $\mathcal{T}_z[\lambda_{int}]= I + II + III + IV + V$, the constant term, i.e., the coefficient of $1$, is
$$-(n-1)\frac{F_z}{z}+aA^2-\frac{A^4F_{zz}}{A^4+F_z^2}=0$$
in light of (\ref{eq:ODE-F}). On the other hand, the coefficient of the term $\tau e^{-\tau}$ in the expression for $\mathcal{T}_z[\lambda^+_{int}]$ is
\begin{equation}
\begin{split}
&~~~~ -B-\frac{DQ_z}{z}(n-1)-2aAB-\frac{1}{A^4+F_z^2}\left(-A^4F_{zz}\frac{2DF_zQ_z-4A^3B}{A^4+F_z^2}+A^4DQ_{zz}-4A^3BF_{zz}\right) \\
&= -B-2aAB-\frac{D}{z}(n-1)Q_z+\frac{A^4F_{zz}(2DF_zQ_z-4A^3B)}{(A^4+F_z^2)^2}-\frac{A^4DQ_{zz}}{A^4+F_z^2}+\frac{4A^3BF_{zz}}{A^4+F_z^2} \\
&= -B-2aAB+\frac{4A^3BF_z^2F_{zz}}{(A^4+F_z^2)^2}+D\left(-\frac{Q_z}{z}(n-1)-\left[\frac{Q_z}{1+F_z^2/A^4}\right]_z\right).
\end{split}
\end{equation}
So we choose $Q$ to be the unique function satisfying 
\begin{equation}
\label{eq:ODE-Q}
-\frac{Q_z}{z}(n-1)-\left[\frac{Q_z}{1+F_z^2/A^4}\right]_z=1, ~~Q(0)=Q_z(0)=0
\end{equation} 
which is a smooth even function. 

It follows from the asymptotic expansion \eqref{eq:asymp-F} of $F$ that there exists a constant $C=C(R_1)>0$ such that for $|z|\leqslant R_1$,
\begin{align*} 
\left\vert F\right\vert,\left\vert zF_z \right\vert & \leqslant C, \\
\displaystyle{ \left\vert \frac{A^{-5}F_z^2F_{zz}}{(1+F^2_z/A^4)^2} \right\vert} & \leqslant C.
\end{align*} 
Then using equations \eqref{eq:ODE-F} and \eqref{eq:ODE-Q} satisfied by $F$ and $Q$ respectively, we have for $|z|\leqslant R_1$ and $\tau\geqslant \tau_1$ with $\tau_1$ sufficiently large,
\begin{align*}
e^{\tau}\mathcal{T}_z[\lambda^+] & = \left(D^+ - (1+2aA^+) B^+ + \frac{4(A^+)^3B^+(F^+_z)^2F^+_{zz}}{((A^+)^4+(F^+_z)^2)^2}\right)\tau +O(1)\\
& \geqslant \left(D^+ - \left(1+2aA^+ + 4C\right)|B^+| \right)\tau + O(1)\\
& > 0
\end{align*}
where the last inequality holds so long as $D^+ > (1+2aA^+ + C)|B^+|$.

Similarly, for $D^{-} < -(1+2aA^{-} +4C)|B^{-}|$, we have $\mathcal{T}_z[\lambda^-]<0$ for $|z|\leqslant R_1$ and $\tau\geq\tau_1$. The lemma is therefore proven.

\end{proof}

\subsection{Exterior region}

In the exterior region, we work with the quantity $\lambda(\phi,\tau)$, and with the corresponding MCF equation \eqref{eq:lambda(phi,tau)}. Hence, we define the quasilinear parabolic operator
\begin{align}\label{eq:Fphi}
\mathcal{F}_\phi[\lambda] &: =  \left. \p_\tau \right\vert_\phi \lambda - \frac{\lambda_{\phi\phi}-2\lambda^2_\phi/\lambda}{1+e^{\tau}\lambda^2_\phi/\lambda^4} - \left(\frac{n-1}{\phi}-\frac{\phi}{2}\right)\lambda_\phi + a \lambda^2.
\end{align}
For the equation $\mathcal{F}_\phi[\lambda]=0$, we seek subsolutions and supersolutions, whose existence is proven in the following lemma.

\begin{lemma}\label{exterior-supersub}
For an integer $n\geqslant 2$ and positive constants $c^{\pm}$ such that $c^{\pm}-a\log(2n-2)>0$, we define\footnote{This definition is consistent with \eqref{eq:lambdabar}; therefore $\bar\lambda$ satisfies equation \eqref{eq:ODE-lambdabar}.}
\begin{align}
\label{eq:barlambda}
\bar\lambda^{\pm} &= \bar\lambda^{\pm}(\phi) := \frac{-1}{c^{\pm} - a \log(2n-2-\phi^2)}.
\end{align}
There exists an even function 
$$\psi:(-\sqrt{2(n-1)},\sqrt{2(n-1)}) \to \RR$$
such that for any fixed $R_2>0$, there exist a pair of constants $b^{\pm}$ and sufficiently large $\tau_2<\infty$, the functions 
\begin{align}
\label{eq:ext-supersub}
\lambda^{\pm}_{ext}(\phi,\tau) := \bar\lambda^{\pm}(\phi) + b^{\pm} e^{-\tau}\psi(\phi)
\end{align}
are a supersolution ($+$) and a subsolution ($-$) respectively, of $\mathcal{F}_\phi[\lambda]=0$ over the region $R_2 e^{-\tau/2} \leqslant \vert\phi\vert < \sqrt{2(n-1)}$ for all $\tau\geq\tau_2$. The constant $b^{\pm}$ depends on $n, R_2$ and $c^{\pm}$ respectively.

\end{lemma}

\begin{proof}

In the proof, we omit the $\pm$ in the notations as the calculation is uniform for both. The difference only appears at the end of the argument, as we see below. The functions involved are all even in $\phi$, so we need only consider $\phi\geqslant 0$.

Applying the operator $\mathcal{F}_\phi$ defined in \eqref{eq:Fphi} to the function $\lambda_{ext}$ from \eqref{eq:ext-supersub}, we obtain 
\begin{align*}
e^{\tau}\mathcal{F}_\phi[\lambda_{ext}] & = I + II + III + IV,
\end{align*}
where for simplicity, we use $\lambda$ in place of $\lambda_{ext}$ and have 
\begin{align*}
I & = e^{\tau} \left. \p_\tau \right\vert_\phi \lambda  = - b \psi,\\
II & = - \frac{(\lambda_{\phi\phi}-2\lambda^2_\phi/\lambda)}{e^{-\tau}+\lambda^2_\phi/\lambda^4},\\
III & = - e^{\tau} \left(\frac{n-1}{\phi}-\frac{\phi}{2}\right)\lambda_\phi \\
& = - e^{\tau}\left(\frac{n-1}{\phi}-\frac{\phi}{2}\right) \bar\lambda' - \left(\frac{n-1}{\phi}-\frac{\phi}{2}\right) b\psi',\\
IV & = e^{\tau}a \lambda^2 \\
& = e^{\tau}a \left( \bar\lambda +  b e^{-\tau} \psi \right)^2\\
& = e^{\tau} a \bar\lambda^2 + 2ab\bar\lambda\psi + e^{-\tau} ab^2\psi^2.
\end{align*}
where $\bar\lambda$ solves equation \eqref{eq:ODE-lambdabar}. Using \eqref{eq:ODE-lambdabar} and combining, we have 
\begin{align*}
e^{\tau}\mathcal{F}_\phi[\lambda] & = II + b\left[ (2a\bar\lambda - 1)\psi - \left(\frac{n-1}{\phi}-\frac{\phi}{2}\right) \psi' \right] + e^{-\tau} ab^2\psi^2.
\end{align*}
If we define 
\begin{align*}
\Lambda(\phi) &:= - \frac{\bar\lambda'' - 2(\bar\lambda')^2/\bar\lambda}{(\bar\lambda')^2/\bar\lambda^4},
\end{align*}
then it follows from \eqref{eq:barlambda} that
\begin{align*}
\Lambda(\phi) & = - \frac{2n-2+\phi^2}{2a\phi^2} \frac{1}{(c - a \log(2n-2-2-\phi^2))^2}\\
& = - \frac{2n-2+\phi^2}{2a\phi^2} \bar\lambda^2.
\end{align*}
We note that $\Lambda < 0$ for $0\leqslant\phi<\sqrt{2(n-1)}$. 

We now take the function $\psi(\phi)$ to be any solution of the ODE
\begin{align}
\label{eq:ODE-psi}
(2a\bar\lambda - 1) \psi - \left( \frac{n-1}{\phi} - \frac{\phi}{2} \right)\psi' & = \Lambda.
\end{align}
This ODE can be solved explicitly and the general solution $\psi$ is 
\begin{align}
\label{eq:soln-psi}
\psi & = \bar\lambda^2 \left(\frac{1}{a} + C_1 (2n-2-\phi^2) +  \frac{2n-2-\phi^2}{4a(n-1)}\left(\log(\phi^2) - \log(2n-2-\phi^2)\right) \right)
\end{align}
for an arbitrary positive constant $C_1$. It follows that
\begin{align*}
e^{\tau}\mathcal{F}_\phi[\lambda_{ext}] & = II + b \Lambda(\phi) + e^{-\tau} ab^2\psi^2.
\end{align*}
We now estimate the term $II$. It follows from \eqref{eq:ext-supersub}) that 
\begin{align*}
\lambda & = \bar\lambda \left(1 + e^{-\tau} b\psi/\bar\lambda\right),\\
\lambda_\phi & = \bar\lambda' \left( 1 + e^{-\tau} b\psi'/\bar\lambda'\right),\\
\lambda_{\phi\phi} & = \bar\lambda'' \left( 1 + e^{-\tau} b\psi''/\bar\lambda''\right),
\end{align*}
and so we need to estimate the terms $\psi/\bar\lambda$, $\psi'/\bar\lambda'$, and $\psi''/\bar\lambda''$. We  do this  by considering the asymptotics near $\phi=0$ and near $\phi=\sqrt{2(n-1)}$, respectively.

We first consider the asymptotics as $\phi\nearrow \sqrt{2(n-1)}$. From \eqref{eq:barlambda} and \eqref{eq:soln-psi}, we have as 
$\phi\nearrow \sqrt{2(n-1)}$ the following asymptotics  
\begin{align*}
\psi/\bar \lambda & = \bar\lambda \left(\frac{1}{a}+o(1)\right),\\
\psi'/\bar \lambda' & = \bar\lambda \left(\frac{2}{a}+o(1)\right),\\
\psi''/\bar \lambda'' & = -\frac{1}{a^2} + \bar\lambda\left(\frac{2}{a} + o(1)\right),
\end{align*}
where $\bar\lambda\to 0^-$ as $\phi\nearrow\sqrt{2(n-1)}$. These asymptotics imply that for some fixed $\delta>0$ (e.g., $\delta=1/4$), if $ \delta\leqslant\phi < \sqrt{2(n-1)}$, then there exists a constant $M_1$ independent of $\tau$ such that 
\begin{align}\label{eq:bound-M1}
\left\vert \psi \right\vert, \left\vert \frac{\psi}{\bar\lambda}\right\vert, \left\vert \frac{\psi'}{\bar\lambda'}\right\vert, \left\vert \frac{\psi''}{\bar\lambda''}\right\vert \leqslant M_1.
\end{align}
We also have for $\phi^2<2n-2$, 
$$\left\vert\frac{\psi^2}{\Lambda}\right\vert= \left\vert\frac{2a\phi^2}{2n-2+\phi^2}\right\vert\cdot\left\vert\frac{\psi^2}{\bar\lambda^2}\right\vert\leqslant a\left\vert\frac{\psi^2}{\bar\lambda^2}\right\vert.$$

By direct calculation, we have  
\begin{align*}
II & = - \frac{\lambda_{\phi\phi}-2\lambda^2_\phi/\lambda}{e^{-\tau} + \lambda^2_\phi/\lambda^4} \\
& = \Lambda(\phi) \left( 1 + O\left(e^{-\tau}b\psi/ \bar\lambda, e^{-\tau}b\psi'/\bar\lambda', e^{-\tau}b\psi''/ \bar\lambda'' \right) \right).
\end{align*}
Now fix some $\delta>0$, and we see that by the above asymptotics, for $ \delta \leqslant \phi < \sqrt{2(n-1)}$,
\begin{align}\label{eq:estimate-pi3}
\left\vert O\left(e^{-\tau}b \psi/ \bar\lambda, e^{-\tau}b \psi'/\bar\lambda', e^{-\tau}b \psi''/ \bar\lambda'' \right) \right\vert \leqslant b M_2 e^{-\tau}
\end{align}
for some constant $M_2$. \\

For the supersolution, recalling $\Lambda<0$, we have
\begin{align*}
e^{\tau}\mathcal{F}_\phi[\lambda^+_{ext}] & = II + b^+ \Lambda(\phi) + e^{-\tau} a(b^+)^2\psi^2\\
& \geqslant \Lambda \left(b^+ + 1 +  \left(M_2 b^+  + a^2(b^+)^2M_1^2\right)e^{-\tau} \right).
\end{align*}
If we choose $\tau_2$ sufficiently large so that $\left(M_2 b^++a^2(b^+)^2M_1^2\right)e^{-\tau}<1$ for $\tau\geq\tau_2$, then for $ \delta \leqslant \phi < \sqrt{2(n-1)}$, we have for $b^+ < - 2$, 
\begin{align*}
e^{\tau}\mathcal{F}_\phi[\lambda^+_{ext}] & > \Lambda \left( b^+ + 2  \right) > 0.
\end{align*} 

Next, we consider the asymptotics as $\phi\searrow 0$. Using \eqref{eq:barlambda} and \eqref{eq:soln-psi}, we have as $\phi\searrow 0$ the following asymptotics
\begin{align*}
\psi/\bar \lambda & = \bar\lambda\left(\frac{1}{a}\log\phi + \frac{1}{a}+(2n-2)C_1-\frac{1}{2a}\log(2n-2) + O\left(\phi^2 \log(\phi^2)\right) \right),\\
\psi'/\bar \lambda' & = \phi^{-2}\left(\frac{(n-1)}{a^2} + O\left(\phi^2 \log(\phi^2)\right)\right),\\
\psi''/\bar \lambda'' & = \phi^{-2}\left(-\frac{n-1}{a^2}+O\left(\phi^2 \log(\phi^2)\right)\right)
\end{align*}
These asymptotics imply that for $0<\phi \leqslant \delta$, there exists a constant $M_3$ (independent of $\tau$) such that
\begin{align}
\label{eq:bound-M2}
\left\vert \psi\right\vert, \left\vert \frac{\psi}{\bar \lambda} \right\vert\leqslant M_3(-\log\phi),\quad\quad \left\vert \frac{\psi'}{\bar \lambda'} \right\vert, \left\vert \frac{\psi''}{\bar \lambda''} \right\vert \leqslant M_3 \phi^{-2},
\end{align}
and hence we have
\begin{align*}
II & = - \frac{\lambda_{\phi\phi}-2\lambda^2_\phi/\lambda}{e^{-\tau} + \lambda^2_\phi/\lambda^4} \\
& = \Lambda(\phi) \left( 1 +  O(e^{-\tau}b \psi/ \bar\lambda, e^{-\tau}b \psi'/\bar\lambda', e^{-\tau}b \psi''/ \bar\lambda'' )  \right).
\end{align*}

From the known estimates, we have
\begin{align}\label{eq:estimate-pi4}
\left\vert O(e^{-\tau}b \psi/ \bar\lambda, e^{-\tau}b \psi'/\bar\lambda', e^{-\tau}b \psi''/ \bar\lambda'' ) \right\vert \leqslant M_4 \phi^{-2} e^{-\tau},\\
\left\vert \frac{e^{-\tau} ab^2\psi^2}{\Lambda}\right\vert  \leqslant e^{-\tau}b^2 M_4\left( \phi^2(\log\phi)^2 + O(\phi^3(\log\phi)^2) \right).
\end{align}
for some constant $M_4$. It then follows that for $0<R_2\leqslant\phi e^{\tau/2}$, 
\begin{align}
\label{eq:R2}
M_4\phi^{-2}e^{-\tau} & \leqslant M_4 R_2^{-2} .
\end{align}
Now for the supersolution, if we choose $\tau_2$ to be even larger so that for $\tau\geq\tau_2$ and for $R_2e^{-\tau/2}\leqslant \phi<\delta$, we have for such $\phi$ that
\begin{align*}
\left\vert M_4\left( \phi^2(\log\phi)^2 + O(\phi^3(\log\phi)^2) \right)\right\vert \leqslant M_5, 
\end{align*}
for a constant $M_5$ (independent of $\tau$), and therefore
\begin{align*}
e^{\tau}\mathcal{F}_\phi[\lambda^+_{ext}] & = II + b^+ \Lambda(\phi) + e^{-\tau}a(b^+)^2\psi^2\\
& > \Lambda(\phi) \left( b^+ +  1 + e^{-\tau_2}M_5 (b^+)^2 \right) \\
& > 0
\end{align*}
as long as $b^+ < - \left( 1 + e^{-\tau_2}M_5 (b^+)^2\right)$. Indeed, if we choose $\tau_2$ sufficiently large so that $e^{-\tau_2}M_5\ll 1$, then $e^{-\tau_2}M_5 (b^+)^2 + b^+ + 1 =0$ has two real solutions $b_1^+$ and $b_2^+$. (In fact, for the supersolution, we can simply drop the term $e^{-\tau}a(b^+)^2\psi^2$.)

Therefore, if we take $b^{+} < \min \{-2, b_1^+, b_2^+\} < 0$, then $\lambda^{+}_{ext}$ is a supersolution of $\mathcal{F}_\phi[\lambda]=0$ over $R_2 e^{-\tau/2} \leqslant \phi < \sqrt{2(n-1)}$ for all $\tau\geq\tau_2$. \\

By a similar argument, let $b_1^-$, $b_2^-$ be solutions to the quadratic equation $e^{-\tau_2}M_5 (b^-)^2 - b^- -1 =0$ (which always has real solutions) so that $b^- > - \left( 1 - e^{-\tau_2}M_5 (b^+)^2\right)$, and take $b^{-} \geqslant \max \{-1/2, b_1^-, b_2^-\}$. Then we have that $\lambda^{-}_{ext}$ is a subsolution of $\mathcal{F}_\phi[\lambda]=0$ over $R_2 e^{-\gamma\tau}\leqslant \phi < \sqrt{2(n-1)}$ for all $\tau\geqslant\tau_2$.

The lemma is proven.

\end{proof}

\begin{remark}
It follows from the proof of Lemma \ref{exterior-supersub} that we can pick $b^- > 0$. This is convenient for considerations below.
\end{remark}

\section{Upper and lower barriers}\label{barriers} 

According to Lemmata \ref{interior-supersub} and \ref{exterior-supersub}, if we choose $R_2<R_1$, then there is an overlap of the interior and exterior regions where both $\lambda^{\pm}_{int}$ and $\lambda^{\pm}_{ext}$ are defined. In order to show that the regional supersolutions $\lambda^+_{ext}$ and $\lambda^+_{int}$ together with the regional subsolutions $\lambda^-_{ext}$ and $\lambda^-_{int}$ collectively provide upper and lower barriers by the standard $\sup$ and $\inf$ constructions for our mean curvature flow problem, we need to show the following:
\begin{itemize}

\item[(i)] in each region, $\lambda ^-_{int}\leqslant \lambda ^+_{int}$ and $\lambda ^-_{ext}\leqslant \lambda ^+_{ext}$;

\item[(ii)] $\lambda ^+_{int}$ and $\lambda ^+_{ext}$ patch together; i.e., $\sup\{\lambda ^+_{int}, \lambda ^+_{ext}\}$ takes the values of $\lambda ^+_{int}$ and then $\lambda ^+_{ext}$ in moving from the interior to the exterior region. Similarly for $\lambda ^-_{int}$ and $\lambda ^-_{ext}$;

\item[(iii)] the patched supersolutions and subsolutions have the desirable comparison relation throughout; i.e., $\lambda ^-_{ext}\leqslant \lambda ^+_{int}$ and $\lambda ^-_{int}\leqslant \lambda ^+_{ext}$ in the overlapping region where all of them are defined. 

\end{itemize}

We first prove (i), via the following two lemmata. 

\begin{lemma}
\label{int-order}
For $A^{-}>A^{+}$, there exists $\tau_3\geq\tau_1$, where $\tau_1$ is defined in Lemma \ref{interior-supersub}, such that
\begin{align*}
\lambda^{\pm}_{int} = -A^{\pm} + e^{-\tau}F^{\pm}(z) + \left(B^{\pm}\tau+E^{\pm}\right) e^{-\tau} +  \tau e^{-2\tau}D^{\pm}Q^{\pm}(z)
\end{align*}
satisfy $\lambda^{-}_{int} < \lambda^{+}_{int}$ for $|z|\leqslant R_1$ and for $\tau\geq\tau_3$.
\end{lemma}

\begin{proof}
$F^{\pm}$ and $Q^{\pm}$ are bounded on $|z| \leqslant R_1$. Since $A^{-}>A^{+}$, we have	
\begin{align*}
\lambda^{+}_{int} - \lambda^{-}_{int} & = A^{-} - A^{+} +  (B^{+}-B^{-})\tau e^{-\tau} + (E^{+} - E^{-})e^{-\tau} + O(e^{-\tau}) \\
& = A^{-} - A^{+} + O(\tau e^{-\tau})\\
& > 0 
\end{align*}
for $\tau\geqslant \tau_3$ sufficiently large (larger than $\tau_1$ if necessary).
\end{proof}

\begin{lemma}
\label{ext-order}
For $c^{+}>c^{-}$, there exists $\tau_4\geq\tau_1$ such that
\begin{align*}
\lambda^{\pm}_{ext} & = \bar\lambda^\pm(\phi) + b^{\pm} e^{-\tau}\psi(\phi)
\end{align*}
(as in Lemma \ref{exterior-supersub}) satisfy $\lambda^{-}_{ext}< \lambda^{+}_{ext}$ for $R_2e^{-\tau/2}\leqslant |\phi| <\sqrt{2(n-1)}$ and $\tau\geq\tau_4$.
\end{lemma}

\begin{proof}
Based on the formulas \eqref{eq:ext-supersub} for $\lambda ^+_{ext}$ and $\lambda^-_{ext}$, we have 
\begin{align*}
\lambda^{+}_{ext} - \lambda^{-}_{ext} & = \bar\lambda^+(\phi)-\bar\lambda^-(\phi) + b^+ e^{-\tau} \psi^+-b^-e^{-\tau}\psi^-.
\end{align*}
By the expressions \eqref{eq:barlambda} for $\bar\lambda$ and \eqref{eq:soln-psi} for $\psi$, for a small $\delta>0$, we have
\begin{enumerate}

\item[(a)] For $R_2 e^{-\tau/2}\leqslant |\phi|\leqslant \sqrt{2(n-1)}-\delta$, there is $\epsilon>0$, such that 
$$\bar\lambda^+ -\bar\lambda^->\epsilon,$$
$$\psi^{\pm}=(\bar\lambda^{\pm})^2\cdot O(\tau),$$ 
and so $\lambda_{ext}^+-\lambda_{ext}^->0$ for large $\tau$. 

\item[(b)] For $\sqrt{2(n-1)}-\delta<\phi<\sqrt{2(n-1)}$, by the expansion of $\psi/\bar\lambda$ near $\sqrt{2(n-1)}$ in the proof of Lemma \ref{exterior-supersub}, we have 
$$\psi^{\pm}=(\bar\lambda^\pm)^2\left(\frac{1}{a}+o(1)\right)$$
and by simple calculation, for $\bar\lambda^+>\bar\lambda^-$ both negative but close to $0$, we have 
$$\bar\lambda^+-\bar\lambda^-\geqslant C|\bar\lambda^{\pm}|^2.$$
So $\lambda_{ext}^+-\lambda_{ext}^->0$ for large $\tau$. 

\end{enumerate}
The lemma then follows by taking $\tau_4$ sufficient large.

\end{proof}

Now, we move on to justify (ii), i.e., the gluing by taking supremum and infimum. Recall that Lemma \ref{interior-supersub} holds for any $R_1>0$ and Lemma \ref{exterior-supersub} holds for any $R_2>0$. Below, we  choose $1\ll R_2<R_1$ and patch together $\lambda ^+_{int}$ and $\lambda ^+_{ext}$, and $\lambda ^-_{int}$ and $\lambda ^-_{ext}$ in the region defined by $\{ R_2 < z < R_1 \}$. To this end, we need the following lemma.

\begin{lemma}\label{patch} 
For a fixed integer $n\geqslant 2$, let $\lambda^{+}_{int}$ and $\lambda^{-}_{int}$ be as discussed in Lemmata \ref{interior-supersub} and \ref{int-order}, and $\lambda^{+}_{ext}$ and $\lambda^{-}_{ext}$ as discussed in Lemmata \ref{exterior-supersub} and \ref{ext-order}. There are properly chosen constants $A^{\pm}>0$, $c^{\pm}>0$, $B^{+},b^->0$, and $B^-, b^+ <0$ satisfying
\begin{align}
A^{\pm} & = \frac{1}{c^{\pm}-a\log(2n-2)}, \\
B^{\pm} & = -\frac{b^\pm}{2a}(A^\pm)^2,
\end{align}
such that for sufficiently large $R_1$ and $R_2$ for Lemmata \ref{interior-supersub} and \ref{exterior-supersub}, we have for $\tau\geqslant \tau_5$ with some sufficiently large $\tau_5$ that the pair of functions
$$\lambda^{+}_{int} - \lambda^{+}_{ext},\quad\text{and} \quad \lambda^{-}_{ext} - \lambda^{-}_{int}$$ 
both strictly increase from negative to positive in the $z$-interval $(R_2, R_1)$. 

\end{lemma}

\begin{proof}

As in the proofs of Lemmata \ref{interior-supersub} and \ref{exterior-supersub}, we prove this lemma for $\phi \in [0, \sqrt{2(n-1)})$; the proof for negative values of $\phi$ follows from evenness. Furthermore, we only consider a bounded $z$ interval, on which $\phi$ is very close to $0$ for large $\tau$.   

In the interior region, using the asymptotic expansion of $F(z)$ in \eqref{eq:asymp-F}, we have that as $z\to \infty$,
\begin{align*}
\lambda^{+}_{int} & = - A^{+} + B^{+}\tau e^{-\tau} \\
& \quad + e^{-\tau}\left(\frac{a (A^{+})^2}{2(n-1)}z^2 - \frac{(A^+)^2}{a}\log(az) + E^{+} + O(z^{-2})\right)\\
& \quad +  D^{+}\tau e^{-2\tau}Q^{+}(z).
\end{align*}

In the exterior region, using the asymptotic expansion that readily follows from the explicit expression for $\psi(\phi)=\psi\left(z e^{-\tau/2}\right)$ in \eqref{eq:soln-psi} and denoting $\alpha^{+}:= [c^{+}-a\log(2n-2)]^{-1}$, we have that for $\phi$ near $0$, 
\begin{align*}
\lambda^{+}_{{ext}} & = \bar\lambda^{+}\left(z e^{-\tau/2}\right) + b^{+}e^{-\tau}\psi\left(z e^{-\tau/2}\right)\\
& = -\alpha^{+} + \frac{a(\alpha^+)^2}{(2n-2)} z^2 e^{-\tau} + O\left(z^4e^{-2\tau}\right)\\
& \quad + e^{-\tau}b^+\left((\alpha^+)^2 - \frac{a(\alpha^+)^3}{(n-1)}z^2 e^{-\tau} \right)\cdot\left( d + \frac{1}{a}\log|z| - \frac{1}{2a}\tau \right)\\
&\quad + O\left( z^2 e^{-2\tau}\left(1+\log|z|+\tau\right)\right)\\
& = -\alpha^+ - \frac{b^+ (\alpha^+)^2}{2a}\tau e^{-\tau}\\
& \quad + e^{-\tau}\left(\frac{a(\alpha^+)^2}{2(n-1)}z^2 + \frac{(\alpha^+)^2b^+}{a}\log|z| + (\alpha^+)^2 b^+ d \right)\\
& \quad + O\left( z^2 e^{-2\tau}\left(1+\log|z|+\tau\right)\right),
\end{align*}
where $d=\frac{1}{a}+C_1(2n-2)+\frac{1}{2a}\log(2n-2)>0$. 

It then follows that
\begin{align*}
\lambda^{+}_{int} - \lambda^{+}_{ext} & = \left( \alpha^{+} - A^{+} \right) + \left(B^{+} +\frac{b^+ (\alpha^+)^2}{2a} \right)\tau e^{-\tau}\\
& \quad + e^{-\tau}\left(\frac{a (A^+)^2}{2(n-1)}z^2 - \frac{a (\alpha^+)^2}{2(n-1)}z^2 - \frac{(A^+)^2}{a}\log(az) - \frac{(\alpha^+)^2b^+}{a}\log|z| \right)\\
& \quad + e^{-\tau}\left( E^+ - (\alpha^+)^2b^+d + O\left(z^{-2}\right) \right)\\
& \quad + O\left(\tau e^{-2\tau}Q^{+}(z)\right) +  O\left( z^2 e^{-2\tau}\left(1+\log|z|+\tau\right)\right).
\end{align*}

Now choose $A^+$, $B^+$, $c^+$ and $b^+$ satisfying 
\begin{align}
A^{+} & = \alpha^+ =  \frac{1}{c^+ - a\log(2n-2)}>0, \label{eq:Acplus}\\
B^+ & = -\frac{b^+}{2a} \left(A^+\right)^2>0,
\end{align}
noticing that $b^+<0$, then the constant terms and the $\tau e^{-\tau}$ terms are eliminated. So it follows that
\begin{align*}
e^{\tau}(\lambda^{+}_{int} - \lambda^{+}_{ext}) \\
& = \frac{(A^+)^2}{a}(-b^+-1)\log|z|\\
&\quad + E^+ - (A^+)^2\left(\frac{\log a}{a} + b^+d\right)\\
& \quad + O\left(z^{-2}\right)\\
& \quad + O\left(\tau e^{-\tau}Q^{+}(z)\right) +  O\left( z^2 e^{-\tau}\left(1+\log|z|+\tau\right)\right).
\end{align*}
The derivative of $e^{\tau}(\lambda^{+}_{int} - \lambda^{+}_{ext})$ with respect to $z$ is given by 
\begin{align*}
e^{\tau}(\lambda^{+}_{int} - \lambda^{+}_{ext})_z
& = \frac{(A^+)^2}{a}(-b^+-1)z^{-1} + O\left(z^{-3}\right)\\
& \quad + O\left(\tau e^{-\tau}(Q^{+})'(z)\right) +  O\left( z e^{-\tau}\left(1+\log|z|+\tau\right)\right).
\end{align*}

So far, we have chosen $a>0$ and $d>0$. In light of the above expressions, for any choice of $A^{_+}$ according to \eqref{eq:Acplus}, Lemma \ref{interior-supersub} holds for $\lambda^{+}_{{int}}$, and we can choose $b^+$ such that $-b^+-1>0$ while Lemma \ref{exterior-supersub} holds for $\lambda^+_{{ext}}$. Consequently, we have the following observations regarding $\lambda^{+}_{int}-\lambda^{+}_{ext}$ for sufficiently large $\tau$:
\begin{enumerate}
\item $e^{\tau}(\lambda^{+}_{int} - \lambda^{+}_{ext})$ is smooth and strictly increasing with respect to $z$ on any interval $(R, 10R)$ where $R\gg 1$. 

\item By adjusting the value of $E^{+}$, which is a constant independent of $\tau$, we can make sure $e^{\tau}(\lambda^{+}_{int} - \lambda^{+}_{ext})$ has only one zero at some $z\in(R, 10R)$ as long as (1) holds.	
\end{enumerate}

Letting $R_2 = R\gg 1$ and $R_1 = 10R$, we have that $\lambda^+_{{int}}-\lambda^+_{{ext}}$ strictly increases from negative to positive in the $z$-interval $(R_2, R_1)$

In the same way, we can deal with $\lambda^-_{int}$ and $\lambda^-_{ext}$. In particular, we can choose the same interval $(R_2, R_1)$ by adjusting the previously chosen one if necessary.

Therefore, the lemma is proved.
\end{proof}
\begin{remark}
In Lemma \ref{patch}, the choices of $R_1$ and $R_2$ are independent of the constants $A^\pm$.
\end{remark}

\begin{remark}
The choices of $A^{\pm}$ and of $c^{\pm}$ in Lemma \ref{patch} are compatible with those in Lemmata \ref{int-order} and \ref{ext-order}.
\end{remark}

We can now patch the regional supersolutions and subsolutions, thereby producing the global supersolutions and subsolutions, which are consequently upper and lower barriers. More precisely, for $|\phi|\in[0,\sqrt{2(n-1)})$ and for $\tau\geq\tau_5$, we define $\lambda^{+}=\lambda^{+}(\phi,\tau)$ by
\begin{align}
\label{eq:upperbarrier}
\lambda^{+} := 
\left\{
\begin{array}{cc}
\lambda^{+}_{int} & |\phi|\leqslant R_2 e^{-\tau/2} \vspace{6pt} \\ 
\inf\left\{\lambda^{+}_{int}, \lambda^{+}_{ext} \right\} & R_2 e^{-\tau/2}< |\phi| < R_1 e^{-\tau/2} \vspace{6pt} \\
\lambda^{+}_{ext} & R_1 e^{-\tau/2} \leqslant |\phi| < \sqrt{2(n-1)}
\end{array}
\right. ,
\end{align}
and similarly we define $\lambda^{-}=\lambda^{-}(\phi,\tau)$ by
\begin{align}
\label{eq:lowerbarrier}
\lambda^{-} := 
\left\{
\begin{array}{cc}
\lambda^{-}_{int} & |\phi|\leqslant R_2 e^{-\tau/2} \vspace{6pt} \\
\sup\left\{\lambda^{-}_{int}, \lambda^{-}_{ext} \right\} & R_2 e^{-\tau/2}< |\phi| < R_1 e^{-\tau/2} \vspace{6pt} \\
\lambda^{-}_{ext} & R_1 e^{-\tau/2} \leqslant |\phi| < \sqrt{2(n-1)}
\end{array}
\right. ,
\end{align}
where the above Lemma \ref{patch} is crucial in justifying the legitimate transition from the interior construction to the exterior construction. Some properties of $\lambda^{\pm}$ are summarised in the following proposition.

\begin{prop}
\label{prop-patch}
For a fixed integer $n\geqslant 2$, let $\lambda^{+}$ and $\lambda^{-}$ be defined as in \eqref{eq:upperbarrier} and \eqref{eq:lowerbarrier} respectively. There exists a sufficiently large $\tau_0$ such that the following hold true for $-\sqrt{2(n-1)}<\phi<\sqrt{2(n-1)}$ and $\tau\geq\tau_0$:
\begin{enumerate}

\item[(B1)] $\lambda^{+}$ and $\lambda^{-}$ are supersolution ($+$) and subsolutions ($-$) for equation \eqref{eq:lambda(phi,tau)} respectively; \\

\item[(B2)] $\lambda^{-} < \lambda^{+}$; \\

\item[(B3)] near $\phi = 0$, $\lambda^{\pm} = \lambda^{\pm}_{int}$, and near $\phi=\sqrt{2(n-1)}$, $\lambda^{\pm} = \lambda^{\pm}_{ext}$; \\

\item[(B4)] for any $\tau\in[\tau_0,\infty)$, $\lim\limits_{|\phi|\nearrow \sqrt{2(n-1)}} \lambda^{\pm} = 0$.

\end{enumerate}

\end{prop}

\begin{proof} 

Take $\tau_0\geqslant \tau_5$, where $\tau_5$ is defined in Lemma \ref{patch}. 

Condition (B1) follows from the standard min-max property of supersolutions and subsolutions and from Lemma \ref{patch}. 

Condition (B3) follows from the definition of $\lambda^{\pm}$. 

Condition (B4) follows from Condition (B3) and $\lim\limits_{|\phi|\nearrow \sqrt{2(n-1)}} \lambda^{\pm}_{ext} = 0$ by definition. 

For Condition (B2), in light of Lemmata \ref{int-order} and \ref{ext-order}, we only need to show that for $R_2 e^{-\tau/2}< |\phi|< R_1 e^{-\tau/2}$ (or equivalently, for $R_2<z<R_1$), 
$$\sup\{\lambda^-_{int}, \lambda^-_{ext}\}<\inf\{\lambda^+_{int}, \lambda^+_{ext}\}.$$
Since $\lambda^-_{int}<\lambda^+_{int}$ and $\lambda^-_{ext}<\lambda^+_{ext}$, it suffices to show that
$$\lambda^-_{ext}<\lambda^+_{int}\quad\text{and}\quad\lambda^-_{int}<\lambda^+_{ext}.$$ 

In order to prove these inequalities, we recycle the calculations appearing in the proof of Lemma \ref{patch} as follows. 
\begin{align*}
\lambda^{+}_{int} - \lambda^{-}_{ext} & = \left( - A^{+} + \frac{1}{c^- - a\log(2n-2)} \right) + \left(B^{+} +\frac{b^-}{2a[c^--a\log(2n-2)]^2} \right)\tau e^{-\tau}\\
& \quad + e^{-\tau}\left(\frac{a (A^+)^2}{2n}z^2 - \frac{b^-}{a[c^--a\log(2n-2)]^2}\log|z| + o(z^2) \right)\\
& \quad \quad + e^{-\tau}\left(E^{+} - \frac{b^-d}{[c^--a\log(2n-2)]^2}  \right) +O(z^2e^{-\tau}) \\ 
& \quad \quad \quad + b^- O(\tau z^2e^{-2\tau}) + b^- O(e^{-2\tau}z^2\log z).
\end{align*}
\begin{align*}
\lambda^{-}_{int} - \lambda^{+}_{ext} & = \left( - A^{-} + \frac{1}{c^+ - a\log(2n-2)} \right) + \left(B^{-} +\frac{b^+}{2a[c^+-a\log(2n-2)]^2} \right)\tau e^{-\tau}\\
& \quad + e^{-\tau}\left(\frac{a (A^-)^2}{2n}z^2 - \frac{b^+}{a[c^+-a\log(2n-2)]^2}\log|z| + o(z^2) \right)\\
& \quad \quad + e^{-\tau}\left(E^{-} - \frac{b^+d}{[c^+-a\log(2n-2)]^2}  \right) +O(z^2e^{-\tau}) \\ 
& \quad \quad \quad + b^+ O(\tau z^2e^{-2\tau}) + b^+ O(e^{-2\tau}z^2\log z).
\end{align*}
The constants satisfy 
$$c^+ >c^- >a\log (2n-2),$$
$$0<A^+=\frac{1}{c^+-a\log(2n-2)}<\frac{1}{c^--a\log(2n-2)}=A^-;$$
hence, we see that for $z\in (R_2, R_1)$ (so $|z|$ and $|\log z|$ are bounded), the leading (constant) terms have favourable signs and so for $\tau$ sufficiently large, 
$$\lambda^-_{ext}<\lambda^+_{int}\quad\text{and}\quad\lambda^-_{int}<\lambda^+_{ext},$$ 
which concludes the proof of Condition (B2). 

\end{proof}

We now prove a comparison principle for any pair of smooth functions such that one of them is a smooth subsolution 
of equation $\mathcal{F}_\phi[\lambda] = 0$ (cf. \eqref{eq:lambda(phi,tau)}) and the other is a smooth supersolution of the same equation. These functions need not be the subsolution $\lambda^{-}$ or supersolution $\lambda^{+}$ constructed above, but of course, the purpose of this result is to show that they serve as barriers for a solution. 

\begin{prop}
\label{comparison}{(Comparison principle for $\mathcal{F}_{\phi}[\lambda]=0$)}
For a fixed integer $n\geqslant 2$ and any $\bar\tau\in(\tau_0, \infty)$, where $\tau_0$ is arbitrary, suppose that $\zeta^{+}$ and $\zeta^{-}$ are any smooth non-positive supersolutions and subsolutions of the equation $\mathcal{F}_{\phi}[\lambda]=0$ respectively. Assume that
\begin{itemize}

\item[(C1)] $\zeta^{-}(\phi,\tau_0) \leqslant \zeta^{+}(\phi,\tau_0)$ for $\phi\in(-\sqrt{2(n-1)},\sqrt{2(n-1)})$,\\

\item[(C2)] $\zeta^{-}(-\sqrt{2(n-1)},\tau) \leqslant \zeta^{+}(-\sqrt{2(n-1)},\tau)$ for $\tau\in[\tau_0,\bar\tau]$,\\

\item[(C3)] $\zeta^{-}(\sqrt{2(n-1)},\tau) \leqslant \zeta^{+}(\sqrt{2(n-1)},\tau)$ for $\tau\in[\tau_0,\bar\tau]$, 

\end{itemize}

\noindent Then $\zeta^{-}(\phi,\tau) \leqslant \zeta^{+}(\phi,\tau)$ over $ [-\sqrt{2(n-1)},\sqrt{2(n-1)}] \times [\tau_0,\bar\tau]$.

\end{prop}

\begin{proof}

Take any $\epsilon>0$ and define $v:= e^{-\mu \tau}(\zeta^+ - \zeta^{-}) + \epsilon$ for some $\mu>0$ to be chosen. We claim that $v > 0$ on $[-\sqrt{2(n-1)},\sqrt{2(n-1)}] \times [\tau_0,\bar\tau]$. 

To prove this, suppose the contrary. Then it follows from the assumptions (C1)--(C3) and from the continuity of functions over the compact space-time region that there must be a first time $\tau_*\in(\tau_0,\bar\tau)$ and an interior point $\phi_*\in(-\sqrt{2(n-1)},\sqrt{2(n-1)})$ such that
\begin{align*}
v(\phi_*, \tau_*) = 0
\end{align*}
which is the spatial minimum and minimum for time up to $\tau_*$. So at $(\phi_*,\tau_*)$, we have
\begin{align*}
\p_\tau\vert_{\phi} v  \leqslant 0, &  \quad\quad\quad \zeta^{+}_{\phi\phi}  \geqslant \zeta^{-}_{\phi\phi}, \\
\zeta^{+}_{\phi}  = \zeta^{-}_{\phi}, & \quad\quad\quad 0\geqslant \zeta^{-} > \zeta^{+} = \zeta^{-} -\epsilon e^{\mu \tau_*}.
\end{align*}
Consequently at $(\phi_*,\tau_*)$, we have

\begin{align*}
0 
&\geqslant e^{\mu \tau_*} \p_\tau\vert_{\phi} v \\
&= \p_\tau\vert_{\phi} (\zeta^{+} - \zeta^{-})  - \mu(\zeta^{+}-\zeta^{-}) \\
&\geqslant \frac{(\zeta^+_{\phi\phi}-2(\zeta^{+}_\phi)^2/\zeta^+)}{1+e^{\tau_*}(\zeta^{+}_\phi)^2/(\zeta^{+})^4} - \frac{(\zeta^-_{\phi\phi}-2(\zeta^{-}_\phi)^2/\zeta^-)}{1+e^{\tau_*}(\zeta^{-}_\phi)^2/(\zeta^{-})^4} \\
&\quad + \left(\frac{n-1}{\phi_*}-\frac{\phi_*}{2}\right)\left(\zeta^+_\phi - \zeta^-_\phi\right) + a [(\zeta^-)^2 - (\zeta^+)^2] \\
&\quad\quad - \mu(\zeta^{+}-\zeta^{-}) \\
&\geqslant \frac{(\zeta^-_{\phi\phi}-2(\zeta^{+}_\phi)^2/\zeta^+)}{1+e^{\tau_*}(\zeta^{+}_\phi)^2/(\zeta^{+})^4} - \frac{(\zeta^-_{\phi\phi}-2(\zeta^{-}_\phi)^2/\zeta^-)}{1+e^{\tau_*}(\zeta^{-}_\phi)^2/(\zeta^{-})^4} \\
&\quad + a [(\zeta^-)^2 - (\zeta^+)^2] - \mu(\zeta^{+}-\zeta^{-}) \\
&\geqslant \zeta^-_{\phi\phi}\left(\frac{(\zeta^+)^4}{(\zeta^+)^4+e^{\tau_*}(\zeta^{+}_\phi)^2} - \frac{(\zeta^-)^4}{(\zeta^-)^4+e^{\tau_*}(\zeta^{-}_\phi)^2}\right) \\
&\quad -\frac{2(\zeta^{+}_\phi)^2}{\zeta^+\left(1+e^{\tau_*}(\zeta^{+}_\phi)^2/(\zeta^{+})^4\right)} + \frac{2(\zeta^{-}_\phi)^2}{\zeta^-\left(1+e^{\tau_*}(\zeta^{+}_\phi)^2/(\zeta^+)^4\right)} \\
&\quad\quad + a [(\zeta^-)^2 - (\zeta^+)^2] - \mu(\zeta^{+}-\zeta^{-}) \\
& \geqslant (\zeta^+ - \zeta^-) \left(\frac{e^{\tau_*}(\zeta^+_\phi)^2\zeta^-_{\phi\phi}(\zeta^+ + \zeta^-)\left((\zeta^+)^2+(\zeta^-)^2\right)}{\left(e^{\tau_*}(\zeta^+_\phi)^2 + (\zeta^+)^4\right)\left(e^{\tau_*}(\zeta^+_\phi)^2 + (\zeta^-)^4\right)}\right. \\
&\quad + 	\left. \frac{2(\zeta^+_\phi)^2}{\zeta^+\zeta^-(1+e^{\tau_*}(\zeta^{+}_\phi)^2/(\zeta^{+})^4)} - a\left(\zeta^+ + \zeta^-\right) - \mu \right)\\
& = -\epsilon e^{\mu \tau_*} [ (\text{bounded terms independent of $\mu$}) - \mu],
\end{align*}  
where the ``(bounded terms independent of $\mu$)'' arising in the second to the last step come from the smooth non-positive assumption of $\zeta^{\pm}$ for $\tau\in[\tau_0, \bar\tau]$. For fixed $\epsilon>0$, if we choose $\mu$ sufficiently large, then at $(\phi_*,\tau_*)$, 
\begin{align*}
0 & \geqslant \p_\tau\vert_{\phi} v > 0,
\end{align*}
which is a contradiction. Hence, the claim is true. Since $\epsilon>0$ is arbitrary, the proposition follows.
\end{proof}

\begin{remark}\label{comp_rmk}
We point out that our supersolution and subsolution in Proposition \ref{prop-patch} are defined for $\phi\in (-\sqrt{2(n-1)}, \sqrt{2(n-1)})$ and are continuous and piecewise smooth on their domains of definition. In fact, Proposition \ref{comparison} applies in the piecewise smooth setting. See the discussion in Appendix \ref{comp_rmk_appx}.
\end{remark}

We end this section by discussing the relation between the barriers, $\lambda^\pm$ and a formal solution $\widetilde\lambda$. Given a constant $c>0$ suppose $c^\pm$ are chosen such that $c\in(c^-, c^+)$ and that $A=1/\left(c-a\log(2n-2)\right)\in (A^+, A^-)$ where $A^\pm := 1/\left(c^\pm-a\log(2n-2)\right)$. Now consider the following formal solutions defined in the interior and exterior regions respectively for all $\tau\geq\tau_5$, where $\tau_5$ is defined in Lemma \ref{patch}:
\begin{align*}
\widetilde\lambda_{int}(z,\tau) & = -A + e^{-\tau}F(z),   \quad\quad\quad |z|\in[0, R_1];\\
\widetilde\lambda_{ext}(\phi,\tau) & = -\frac{1}{c-a\log(2n-2-\phi^2)}, \quad |\phi|\in[R_2e^{-\tau/2},\sqrt{2(n-1)}).
\end{align*}
We see that (for example, as a consequence of the proofs of Lemmata \ref{int-order} and \ref{ext-order}) for all $\tau\geq\tau_5$,
\begin{align*}
\lambda^{-}_{\text{int}} & < \widetilde\lambda_{int} < \lambda^{+}_{\text{int}}, \quad |z|\in[0, R_1]; \\
\lambda^{-}_{\text{ext}} & < \widetilde\lambda_{ext} < \lambda^{+}_\text{{ext}}, \quad |\phi|\in[R_2e^{-\tau/2},\sqrt{2(n-1)}).
\end{align*}


\section{Proof of theorem \ref{thmmain}}\label{existence}

We have thus far shown that we have barriers for the mean curvature flow equation. In this section, we first prove a lemma which allows us to show that at $t\nearrow T$, the highest curvature of our convex rotationally symmetric MCF solution occurs at the tip.

The (smooth) hypersurface in $\mathbb{R}^{n+1}$ with the rotation profile $r=u(x)>0$ for $x\geqslant x_0$ has the principal curvatures  
$$\kappa_1=\cdots=\kappa_{n-1}=\frac{1}{u(1+u^2_x)^{1/2}}, \quad \kappa_n=-\frac{u_{xx}}{(1+u^2_x)^{3/2}},$$
where the first $n-1$ indices correspond to the rotation and $n$ to the graph direction. One defines $\mathfrak R:=\kappa_n/\kappa_1$.

\begin{lemma}\label{tip}
For a fixed integer $n\geqslant 2$ and any $\bar t\in[t_0,T)$, $T<\infty$, 
for the above complete noncompact convex rotationally symmetric graphical solution $\Gamma_t$ to the MCF, with uniformly bounded curvature for $t_0\leqslant t\leqslant \bar t$ , assuming $\mathfrak{R}\leqslant C$, where $C\geq 1$, for the initial hypersurface $\Gamma_0$, then $\mathfrak{R}\leqslant C$ for $t\in[0,\bar t]$.
\end{lemma}

\begin{proof}
For the hypersurface evolving by the MCF (\ref{eq:u(x,t)}), we have  
$$\mathfrak{R}=-\frac{u\cdot u_{xx}}{1+u^2_x}=(1-n)-u\cdot u_t.$$

Since $\Gamma_0$ is convex and MCF preserves convexity, we have $u_{xx}<0$ and $\mathfrak{R}>0$ for all $t\in [0, \bar t]$. This and the noncompactness of the hypersurface $\Gamma_t$ for all $t$ imply that $u_x>0$.

The following evolution of $\mathfrak{R}$ is derived in \cite{ADS19}: 
$$\mathfrak{R}_t=\frac{\mathfrak{R}_{xx}}{1+u^2_x}-\frac{2u_x}{u(1+u^2_x)}(1-\mathfrak{R})\mathfrak{R}_x+\frac{2u^2_x}{u^2(1+u^2_x)}[(1-\mathfrak{R}^2)+(n-2)(1-\mathfrak{R})].$$
By the boundedness of curvature and uniqueness, the MCF solution preserves rotational symmetry; in particular, $\mathfrak{R}=1$ at the tip, which is an umbilical point, along the flow. It then follows from the maximum principle that
$$(\mathfrak{R}_{\max})_t\leqslant \frac{2u^2_x}{u^2(1+u^2_x)}(1-\mathfrak{R}_{\max})(\mathfrak{R}_{\max}+n-1),$$ 
from which we obtain $\mathfrak{R}_{\max}\leqslant C$ for $t\in [0, \bar t]$; here we use the fact that $\mathfrak{R}_{\max}(0)\leqslant C$. So the lemma is proven. 
\end{proof}

\begin{remark}\label{ratio_R_rmk}
	We discuss the condition $\mathfrak{R}\leqslant C$ on $\Gamma_0$ in Appendix \ref{ratio_R_appx}.
\end{remark}

We note that $\mathfrak{R}\leqslant C$ implies $C+C u^2_x+u\cdot u_{xx}\geqslant 0$. Also, if $i=1, \cdots, n-1$, then for $\kappa^{-1}_i=u(1+u^2_x)^{1/2}$ we have  
\begin{align*}
(\kappa^{-1}_i)_x=u_x(1+u^2_x)^{-1/2}(1+u^2_x+u\cdot u_{xx})\geqslant 0,
\end{align*}
which means that $\kappa_i$, where $1\leqslant i\leqslant n-1$, achieves the maximum at the tip $u=0$. Then Lemma \ref{tip} can be strengthened to the following version, which says that the highest curvature for our convex rotationally symmetric solution of the MCF is always achieved at the tip, as discussed in \cite{IW19}.

\begin{lemma}\label{tip1}
Under the hypotheses of Lemma \ref{tip}, assuming $\mathfrak{R}\leqslant 1$ for $\Gamma_0$, then for any $t\in [0, \bar t]$, $\mathfrak{R}\leqslant 1$ and the maximum curvature $\sup_{\Gamma_t}|h|$ occurs at the tip of the hypersurface. 
\end{lemma}

We now prove the main theorem of the paper. 

\begin{proof}[Proof of Theorem \ref{thmmain}] 

We fix the dimension $n\geqslant 2$. Let $\tau_0\geqslant \tau_5$, where $\tau_5$ is given in Lemma \ref{patch}.

We begin by constructing the initial data for the MCF flow by patching formal solutions in the interior and exterior regions at $\tau=\tau_0$. Given $a>0$, we fix some constant $c>a\log(2n-2)$ (any such $c$ works) and define
\begin{align}\label{constantA}
A:=\frac{1}{c-a\log(2n-2)}.
\end{align} 
We then find constants $c^+$ and $c^-$, e.g., $c^{\pm}:=c\pm\tilde\epsilon$ for any fixed $\tilde\epsilon>0$, such that $c\in(c^{-}, c^{+})$. We now define constants $A^\pm :=1/\left(c^\pm-a\log(2n-2)\right)$) as in Lemma \ref{patch}, then $A\in(A^{+}, A^{-})$ because $c\in(c^{-}, c^{+})$. Recalling $z=\phi e^{\tau/2}$ and denoting
\begin{align*}
C_0:= A-\frac{1}{c-a\log(2n-2-R_1^2e^{-\tau_0})},
\end{align*}
we define
\begin{align}\label{eq:lambda0}
\widehat\lambda_0(\phi) := \left\{
\begin{array}{lr}
-A + e^{-\tau_0}F(z)- e^{-\tau_0}F(R_1) + C_0, \quad 0\leqslant |z|\leqslant R_1, \\
-1/\left(c-a\log(2n-2-\phi^2)\right), \quad R_1e^{-\tau_0/2}\leqslant |\phi|<\sqrt{2(n-1)},
\end{array}
\right.
\end{align}
where $R_1$ is defined in Lemma \ref{patch}. For $|z|\in [0,R_1]$, $\widehat\lambda$ is given by the profile of a bowl soliton, for which $\mathfrak{R}\leqslant 1$ \cite[Lemma 3.5]{ADS19}. For $|z|\geqslant R_1$, it is straightforward (cf. the proof of Lemma \ref{init_smooth} in Appendix \ref{init_smooth_appx}) to verify that there exists constant $C=C(a,c)$ such that $\mathfrak{R}\leqslant CR_1^{-3}$. We choose $R_1$ such that $100CR_1^{-4} < a$; in particular, Lemma \ref{patch} holds for this $R_1$.

For any $\epsilon>0$ sufficiently small, by taking $\tau_0$ large enough, we have $|- e^{-\tau_0}F(R_1) + C_0|< \epsilon/10$ and for $R_2\leqslant |z|\leqslant R_1$, we have $|e^{-\tau_0}F(z)- e^{-\tau_0}F(R_1)| < \epsilon/10$. Then at $\tau=\tau_0$ (taking even larger $\tau_0$ if needed), we have  
\begin{itemize}
\item For $0\leqslant |z| \leqslant R_1$, i.e., $0\leqslant |\phi| \leqslant R_1 e^{-\tau_0/2}$: 
$$\widehat\lambda_0-\epsilon<\lambda_{int}^-< \widehat\lambda_0 <\lambda_{int}^+ < \widehat\lambda_0+\epsilon.$$

\item For $R_2\leqslant |z| \leqslant R_1$, i.e., $R_2 e^{-\tau_0/2} \leqslant |\phi| \leqslant R_1e^{-\tau_0/2}$:
$$\widehat\lambda_0-\epsilon<\lambda_{int}^-,\;\lambda_{ext}^-<\widehat\lambda_0<\lambda_{int}^+,\;\lambda_{ext}^+ < \widehat\lambda_0+\epsilon.$$

\item For $R_1\leqslant |z|$, i.e., $R_1 e^{-\tau_0/2}\leqslant |\phi|<\sqrt{2(n-1)}$:
$$\widehat\lambda_0-\epsilon<\lambda_{ext}^-<\widehat\lambda_0< \lambda_{ext}^+<\widehat\lambda_0+\epsilon$$

\end{itemize}
So in light of (\ref{eq:upperbarrier}) and (\ref{eq:lowerbarrier}), we conclude that for all $|\phi|\in[0,\sqrt{2(n-1)})$, 
\begin{align*}
\lambda^{-}(\phi,\tau_0) < \widehat\lambda_0(\phi) < \lambda^{+}(\phi,\tau_0), \quad 
\left\vert\lambda^+ - \lambda^- \right\vert < 2\epsilon. 
\end{align*}
It follows from the construction that $\widehat\lambda_0$ is continuous and piecewise smooth, and that 
\begin{align*} 
\lim\limits_{|\phi|\nearrow \sqrt{2(n-1)}} \widehat\lambda_0 = 0
\end{align*}

To apply Lemma \ref{tip}, we need the following Lemma, the proof of which is contained in Appendix \ref{init_smooth_appx}.
\begin{lemma}\label{init_smooth}
We can smooth $\widehat\lambda_0$ to obtain a smooth function $\lambda_0$ such that $\lambda^{-}(\cdot,\tau_0) < \lambda_0 < \lambda^{+}(\cdot,\tau_0)$ for $|\phi|\in[0,\sqrt{2(n-1)})$ and $\lim\limits_{|\phi|\nearrow \sqrt{2(n-1)}} \lambda_0 = 0$. Moreover, after rescaling back to the $(x,u)$-coordinates, the function $u(x)$ corresponding to $\lambda_0$ has the following properties:
$$u(x_0)=0,\quad u_x>0,\quad u_{xx}<0\quad \text{and}\quad \mathfrak{R}\leqslant CR_1^{-3}, $$
for a constant $C=C(a,c)$. In particular, we choose $R_1$ such that $100 C R_1^{-4} < a$.
\end{lemma}

Lemma \ref{init_smooth} allows us to apply Lemma \ref{tip}. The construction and smoothing process actually yields an open\footnote{The open condition only applies near where we smooth the corner. The prescribed geometries near the tip and the spatial infinity are unaffected.} set of such smooth functions $\lambda_0$. Moreover, if we vary the parameters $c$ and $R_1$, then we get a family $\mathscr{G}_0$ of distinct smooth complete noncompact convex rotationally symmetric hypersurfaces $\Gamma_0$ that are asymptotic to a cylinder of radius $\sqrt{2(T-t_0)(n-1)}$. Thus by the main result of \cite{SS14}, the MCF starting from a hypersurface $\Gamma\in\mathscr{G}_0$ must have a smooth solution up to the time $T$ which is exactly the vanishing time of the boundary sphere of the defining domain ball under its own MCF. Let $\lambda(\phi,\tau)$ correspond to such a MCF solution. Since $\lambda^{-} \leqslant \lambda \leqslant \lambda^{+}$ on $(-\sqrt{2(n-1)},\sqrt{2(n-1)})$ at $\tau=\tau_0$ and also for $|\phi|=\sqrt{2(n-1)}$, the comparison principle (Proposition \ref{comparison}) implies that the solution is always trapped between the barriers; i.e.,
$$\lambda^{-} \leqslant \lambda \leqslant \lambda^{+}\quad \text{over}\quad(-\sqrt{2(n-1)},\sqrt{2(n-1)})\times [\tau_0, \infty).$$ 
In particular, the asymptotics of $\lambda^{-}$ and $\lambda^{+}$ as $\phi\nearrow\sqrt{2(n-1)}$ imply that 
\begin{align*} 
\lambda(\phi,\tau) \sim -\frac{1}{c-\log(2n-2-\phi^2)}
\end{align*}
as $|\phi|\nearrow\sqrt{2(n-1)}$ for all $\tau\geq\tau_0$. This implies Item (3) of Theorem \ref{thmmain}.

Now we proceed to justify the accurate curvature blow-up rate and the singularity model as stated in Items (1) and (2) of Theorem \ref{thmmain}. To study the behaviour of such a MCF solution near the tip as $\tau\nearrow\infty$, we work with $y(z,\tau)$ instead of $\lambda(z,\tau)$. Recall that $y(z,\tau)$ evolves by equation \eqref{eq:y(z,tau)}. Let $\tilde A = -1/A$. Define $\tilde p(z,\tau)$ by the relation
\begin{align}
y(\phi,\tau) = \tilde A + e^{-\tau} \tilde p(z,\tau).
\end{align}
Then $\tilde p(z,\tau)$ satisfies the PDE, $\mathcal{B}[\tilde p] = 0$ where
\begin{align*}
\mathcal{B}[\tilde p] & = a - \left(\frac{\tilde{p}_{zz}}{1+\tilde{p}^2_z} + \frac{n-1}{z}\tilde p_{z}\right) + e^{-\tau}\left( \left.\p_\tau\right\vert_z \tilde p + z\tilde p_z - \tilde p \right).
\end{align*}

Recall that $\phi(y, \tau)$ and $y(\phi, \tau)$ denote the functions along the flow which are inverse to each other. Define 
\begin{align*}
y^{(0)}(\phi) & :=c-a\log(2n-2-\phi^2),\\
\phi^{(0)}(y) & :=\sqrt{2n-2-e^{\frac{1}{a}(c-y)}}.
\end{align*}
Let $\lambda^{(0)}(\phi) :=  -1/ y^{(0)}(\phi)$. By the uniformity in the construction of the initial hypersurface and the barriers in terms of $\lambda_0$ and $\lambda^{\pm}$ , we have as $\tau\to \infty$, 
$$\lambda(\phi, \tau)\to \lambda^{(0)}(\phi)$$
locally uniformly for $\phi\in [0, \sqrt{2n-2})$ where $\tilde A=c-a\log(2n-2)$ as in Lemma 7.1 of \cite{AV97}. In particular, we obtain uniform closeness to the barriers on the initial hypersurface by direct construction, whereas in \cite{AV97} the estimates use an Exit Lemma (cf. \cite[Lemma 3.1]{AV97}) and the geometric information of the neck region (i.e. perturbing the neck with Hermite polynomials) in their construction, which is not available in our case. Therefore, 
$$y(\phi, \tau)\to y^{(0)}(\phi)$$
locally uniformly for $\phi\in [0, \sqrt{2n-2})$. 

We then prove the following result corresponding to Lemma 7.2 in \cite{AV97}.  

\begin{lemma}[Type-II blow-up] \label{lem:type2}
Recall the function $\tilde P$ defined in \eqref{eq:tildeP} which forms part of a formal solution to MCF. We have the following asymptotic behaviour of $\tilde p$:
\begin{align}
\lim\limits_{\tau\nearrow\infty} \left( \tilde p(z,\tau) - \tilde p(0,\tau) \right) = \frac{1}{a} \tilde P\left( a z \right)
\end{align}
uniformly on compact $z$ intervals.
\end{lemma}

\begin{proof}[Proof of Lemma \ref{lem:type2}]

By the Fundamental Theorem of Calculus and $\tilde P(0)=0$, it is enough to show that $\tilde p_z(z,\tau)$ converges uniformly to $\tilde{P}'\left(a z \right)$ as $\tau\to\infty$ for bounded $z\geqslant 0$. To do this, it is useful to set a new time parameter 
$$s=e^{\tau}.$$ 
In terms of $s$, $\tilde p$ satisfies the PDE
\begin{align}\label{eq:p_s}
\left. \p_s \right\vert_z \tilde p & = \frac{\tilde{p}_{zz}}{1+\tilde{p}^2_z} + \frac{n-1}{z}\tilde p_{z} - a + \frac{1}{s} (\tilde p - z\tilde p_z).
\end{align}
For simplicity of notations, we further define 
\begin{align}\label{eq:def-q}
q(z,s)&:=\tilde{p}_z(z,\tau)
\end{align}
satisfying $\mathcal{P}[q]=0$, where
\begin{align}
\mathcal{P}[q] & = \frac{\p q}{\p s} + \frac{1}{s} z q_z - \frac{\p}{\p z}\left( \frac{q_z}{1+q^2} + \frac{n-1}{z} q \right).
\end{align} 
We need to prove the convergence of $q(z, s)$. Note that equations \eqref{eq:p_s} and \eqref{eq:P[q]} are of the same type as equations (7.13) and (7.14) in \cite{AV97} and equations (5.1) and (5.2) in \cite{IW19}. In particular, the coefficient $1$ of the term $\frac{1}{s}zq_z$ in equation \eqref{eq:P[q]} can be related to the corresponding coefficients $\frac{m-1}{m-2}$ in \cite{AV97} and $\frac{2\gamma+1}{4\gamma}$ in \cite{IW19}, respectively. Indeed, we have
\begin{align*}
\lim\limits_{m\to\infty}\frac{m-1}{m-2}& = 1 = \lim\limits_{\gamma\to 1/2}\frac{2\gamma+1}{4\gamma}. 
\end{align*} Therefore, the rest of the proof in \cite[pp.51--58]{AV97} applies to our case \emph{mutatis mutandis}. For the convenience of readers and independent interest, the argument for the convergence of $q$ is summarised in Appendix \ref{q_conv_appx}. 
\end{proof}

Lemma \ref{lem:type2} implies that a smooth convex MCF solution expressed in terms of $y(z,\tau)$ satisfies the following asymptotics: on a compact $z$ interval (in the interior region), as $\tau\nearrow\infty$, 
\begin{align*}
y (z,\tau) & = \tilde A + e^{-\tau} \tilde{p}(0,\tau) + e^{-\tau} \frac{1}{a} \tilde P\left( a z \right)\\
& = y(0,\tau) + e^{-\tau} \frac{1}{a} \tilde P\left( a z \right).
\end{align*}
So Item (2) of Theorem \ref{thmmain} is proved. This expansion is indeed valid for higher order derivatives in light of higher order estimates involved in the proof of Lemma \ref{lem:type2}.

Item (2) implies that at $t\nearrow T$, our MCF solution necessarily blows up at the rate predicted by the formal solution $e^{-\tau} \tilde P\left( (a z \right)/a$ (cf. Section \ref{formal}), for which $\mathfrak{R}\leqslant 1$ (cf. \cite[Lemma 3.5]{ADS19}) if $z\in[0, R_1]$. In particular, at the tip, $k_1=k_n=a(T-t)^{-1}$. If $z=u(T-t)^{-1}\geqslant R_1$, then
\begin{align*}
\kappa_1=u^{-1}(1+u_x^2)^{-1/2} \leqslant R_1^{-1}(T-t)^{-1}.
\end{align*}
It then follows from Lemmata \ref{init_smooth} and \ref{tip} that $\mathfrak{R}=k_n/k_1\leqslant CR_1^{-3}$, and hence
\begin{align*}
\kappa_n \leqslant CR_1^{-3} k_1 \leqslant  C R_1^{-4}(T-t)^{-1}< a(T-t)^{-1}.
\end{align*}
So the highest curvature of this MCF solutions occurs at the tip and blows up at the rate $(T-t)^{-1}$, thus proving Item (1) of Theorem \ref{thmmain}.

Therefore, Theorem \ref{thmmain} is proven.
\end{proof}

\begin{remark}
For the class of MCF solutions under consideration, the asymptotic cylindrical condition is given by a precise rate which is preserved under MCF. This provides a more accurate asymptotic behaviour towards spatial infinity than that from the main result in \cite{SS14} for this particular class of solutions. 
\end{remark}


\appendix
\section{Proof of Lemma \ref{init_smooth}}\label{init_smooth_appx}
We need to check the following: 
\begin{align*}
u(x_0)=0, ~~u_x>0, ~~u_{xx}<0 \quad \text{and}\quad \mathfrak{R}\leqslant CR_1^{-3}
\end{align*}
We verify these conditions first on the unsmoothed initial data and then on the smoothed initial data. By symmetry, we only consider $z\geqslant 0$.

\begin{itemize}
\item 
In the region $|z|\in[0,R_1]$, the initial condition is defined to be a scaled and translated copy of the bowl soliton, which is a convex non-collapsed ancient solution to MCF. It follows from the geometry of the bowl soliton that for $z\in(0,R_1)$, we have $u_x>0$, $u_{xx}>0$ and $\mathfrak{R}\leqslant 1$. In particular, the bound on the ratio $\mathfrak{R}$ follows from \cite[Lemma 3.5]{ADS19}.

\item In the region for which $|\phi|=(T-t_0)^{1/2}|z|\in [R_1 e^{-\tau_0/2}, \sqrt{2(n-1)})$, we can use $y$ instead of $x$ since it is merely a translation for our consideration. Recall that $\widehat\lambda=-1/y$ and $z=\phi(T-t_0)^{-1/2}=u(T-t_0)^{-1}$, so we have 
\begin{align*}
\widehat\lambda(z)=\widehat\lambda\bigl((T-t_0)^{-1/2}\phi\bigr)\widehat\lambda\bigl((T-t_0)^{-1}u\bigr)=-\frac{1}{y};
\end{align*}
whence it follows that
\begin{align*}
u_x &= u_y = \frac{(T-t_0)^{1/2}}{y^2\widehat\lambda_\phi} = \frac{T-t_0}{y^2\widehat\lambda_z},\\
u_{xx}& = u_{yy} = -\frac{2(T-t_0)^{1/2}}{y^3\widehat\lambda_\phi}-\frac{(T-t_0)^{1/2}\widehat\lambda_{\phi\phi}}{y^4\widehat\lambda^3_\phi} = -\frac{2(T-t_0)}{y^3\widehat\lambda_z}-\frac{(T-t_0)\widehat\lambda_{zz}}{y^4\widehat\lambda^3_z},
\end{align*}
and
\begin{align}
C+Cu^2_x+u\cdot  u_{xx} & = C+Cu^2_y+u\cdot u_{yy} \notag\\
&= C+(T-t_0)\left(\frac{C\widehat\lambda^4}{\widehat\lambda^2_\phi}+\frac{2\phi\widehat\lambda^3}{\widehat\lambda_\phi}
-\frac{\phi\widehat\lambda^4 \widehat\lambda_{\phi}}{\widehat\lambda^3_\phi}\right) \label{eq:R-quantity1}\\
&= C+(T-t_0)^2\left(\frac{C\widehat\lambda^4}{\widehat\lambda^2_z}+\frac{2z\widehat\lambda^3}{\widehat\lambda_z}
-\frac{z\widehat\lambda^4 \widehat\lambda_{zz}}{\widehat\lambda^3_z}\right).  \label{eq:R-quantity2}
\end{align}

In this region, we have 
$$\widehat\lambda=-\frac{1}{c-a\log\bigl(2n-2-(T-t_0)z^2\bigr)}<0$$ 
since $c>a\log(2n-2)$, and so 
$$\widehat\lambda_z=\frac{2az(T-t_0)\widehat\lambda^2}{2n-2-(T-t_0)z^2}>0,$$
and 
\begin{equation}
\begin{split} 
\widehat\lambda_{zz}
&= \frac{2a(T-t_0)\widehat\lambda^2\bigl(2n-2-(T-t_0)z^2\bigr)+8a^2z^2(T-t_0)^2\widehat\lambda^3+4az^2(T-t_0)^2\widehat\lambda^2}{\bigl(2n-2-(T-t_0)z^2\bigr)^2} \\
&= \frac{4(n-1)a(T-t_0)\widehat\lambda^2+8a^2z^2(T-t_0)^2\widehat\lambda^3+2az^2(T-t_0)^2\widehat\lambda^2}{\bigl(2n-2-(T-t_0)z^2\bigr)^2} \\
&= \frac{2a(T-t_0)\widehat\lambda^2 [2(n-1)+(T-t_0)z^2(4a\widehat\lambda+1)]}{\bigl(2n-2-(T-t_0)z^2\bigr)^2}.  
\end{split} \nonumber
\end{equation} 

So we have
\begin{align*}
\mathfrak{R} & = -\frac{u\cdot u_{yy}}{1+u_y^2}\leqslant -u\cdot u_{yy} \leqslant C(a,c)R_1^{-3},\\
u_y&=\frac{T-t_0}{y^2\widehat\lambda_z}>0,\\
u_{yy}
&= -\frac{2(T-t_0)}{y^3\widehat\lambda_z}-\frac{(T-t_0)\widehat\lambda_{zz}}{y^4\widehat\lambda^3_z} \\
&= \frac{\widehat\lambda\bigl(2n-2-(T-t_0)z^2\bigr)}{az}\\
&\quad - \frac{\bigl(2n-2-(T-t_0)z^2\bigr)[2(n-1)+(T-t_0)z^2(4a\widehat\lambda+1)]}{4a^2z^3(T-t_0)} \\
&= \frac{\widehat\lambda\bigl(2n-2-(T-t_0)z^2\bigr)}{az} - \frac{\widehat\lambda\bigl(2n-2-(T-t_0)z^2\bigr)}{az}\\
&\quad  -\frac{\bigl(2n-2-(T-t_0)z^2\bigr)[2(n-1)+(T-t_0)z^2]}{4a^2z^3(T-t_0)} \\
&= -\frac{\bigl(2n-2-(T-t_0)z^2\bigr)[2(n-1)+(T-t_0)z^2]}{4a^2z^3(T-t_0)} \\
&< 0,
\end{align*}
where the key is that the terms involving $\widehat\lambda$ cancel each other. Furthermore, the calculation of (\ref{eq:R-quantity2}) can be continued and we obtain
\begin{equation} 
\begin{split}
C+Cu^2_x+u\cdot u_{xx} 
&\geqslant C-\frac{(T-t_0)\bigl(2n-2-(T-t_0)z^2\bigr)}{2a^2} \\
&\quad\quad\quad +\frac{\bigl(2n-2-(T-t_0)z^2\bigr)\bigl(2n-2-(T-t_0)z^2-2az(T-t_0)\bigr)}{2a^2z\bigl(c-a\log\bigl(2n-2-(T-t_0)z^2\bigr)\bigr)}.
\end{split}
\end{equation}
For any small $\epsilon>0$, if we take $t_0$ sufficiently close to $T$, then we have
$$\vline\frac{(T-t_0)\bigl(2n-2-(T-t_0)z^2\bigr)}{2a^2}\vline<\epsilon,$$
$$\frac{\bigl(2n-2-(T-t_0)z^2\bigr)\bigl(2n-2-(T-t_0)z^2-2az(T-t_0)\bigr)}{2a^2z\bigl(c-a\log\bigl(2n-2-(T-t_0)z^2\bigr)\bigr)}>-\epsilon$$
uniformly in the region. So $C+Cu^2_x+u\cdot u_{xx} >0$. 

\item To smooth the corner at $z=R_1$, we proceed as follows.

At the connecting point $z=R_1$, we have  
$$\widehat\lambda(R_1)=C+e^{-\tau_0}F(R_1)=-\frac{1}{c-a\log\bigl(2n-2-(T-t_0)R_1^2\bigr)}.$$
Consider the first derivative $u_y=\frac{T-t_0}{y^2\widehat\lambda_z}$. For $z\to R_1^-$,  
$$\widehat\lambda_z=(T-t_0)z\left(\frac{aA^2}{n-1} - \frac{A^2}{az^2}+O\left(z^{-3}\right)\right);$$
then for $z\to R_1^+$, we obtain
\begin{align*}
\widehat\lambda_z & =\frac{2a(T-t_0)\widehat\lambda^2}{2n-2-(T-t_0)z^2}\cdot z\\
& = (T-t_0)z a \widehat\lambda^2 \left( \frac{1}{n-1} + \frac{(T-t_0)z^2}{(2n-2)^2} + O\left( (T-t_0)^2z^4\right)\right).
\end{align*} 
For $t_0 \to T$, since $\widehat\lambda(R_1) =  -A + O(T-t_0)$, 
\begin{align*}
\frac{2a\widehat\lambda^2}{2n-2-(T-t_0)z^2}\to \frac{aA^2}{n-1}>\frac{aA^2}{n-1} - \frac{A^2}{az^2}+O\left(z^{-3}\right).
\end{align*}
Thus by taking any $R_1$ sufficiently large, at the connecting point, $\widehat\lambda$ from the interior is smaller than that from the exterior, and so $u_y$ from the tip side is larger than that from the other side. In fact, the scale of $\widehat\lambda_z$ is $(T-t_0)R_1$, and so the scale of $u_y$ is $\frac{1}{R_1}$. Hence, we can ensure that the graph of $u(y)$ is smooth, increasing and concave; i.e., $u_y>0$ and $u_{yy}<0$. More precisely, fix $\zeta\in(0,1)$ and consider $z\in I_\zeta : = [(1-\zeta) R_1, (1+\zeta)R_1]$. Since
\begin{align*}
u_y|_{z=(1+\zeta)R_1} \leqslant \lim\limits_{z\to R_1^-} u_y\leqslant u_y|_{z=(1-\zeta)R_1},
\end{align*}
we have for $t_0$ sufficiently close to $T$,
\begin{align*}
u_y|_{z=(1-\zeta)R_1} - u_y|_{z=(1+\zeta)R_1} & \geqslant u_y|_{z=(1-\zeta)R_1} - \lim\limits_{z\to R_1^-} u_y, \\
& = \left[a(1+\zeta)R_1 \left( \frac{1}{n-1} + \frac{(T-t_0)(1+\zeta)^2 R_1^2}{(2n-2)^2} + O\left( (T-t_0)^2 R_1^4\right)\right)\right]^{-1} \\
&\quad - \frac{\widehat\lambda^2(R_1)}{\left(\frac{aA^2}{n-1}R_1 - \frac{A^2}{a R_1}+O\left(R_1^{-3}\right) \right)}\\
& \geqslant \left[a(1+\zeta)R_1\left(\frac{1}{n-1} + O(T-t_0)\right)\right]^{-1}\\
& - \quad \frac{\left(A^2 + O(T-t_0)\right)}{\frac{aA^2}{n-1}R_1}\left(1+\frac{1}{a^2R_1^2}+O(R_1^{-3})\right)\\
& = -\frac{\zeta}{1+\zeta}\frac{n-1}{aR_1} + O\left(R_1^{-3}\right) + O\left((T-t_0)R_1^{-1}\right).
\end{align*}

Therefore, on the interval $z=y\sqrt{T-t_0}\in I_\zeta$, we have
\begin{align*}
0> u_{yy}& \geqslant -\frac{1}{2a(1+\zeta)R_1^2} + O\left(R_1^{-4}\right) + O\left((T-t_0)R_1^{-2}\right),
\end{align*}
and hence at $t=t_0$,
\begin{align*}
C+Cu_y^2 + u\cdot u_{yy} & \geqslant C+Cu_y^2|_{z=(1+\zeta)R_1} + u \cdot u_{yy}\\
& \geqslant C - C(\zeta, a, n)R_1^{-2}.
\end{align*}
As $\epsilon$, $a$ and $n$ are fixed, by choosing $R_1$ large enough (noting that in particular both $100 C R_1^{-4}<a$ and Lemma \ref{patch} still hold true), we have
\begin{align*}
C+Cu_y^2 + u\cdot u_{yy} > 0
\end{align*}
on the interval $z=y\sqrt{T-t_0}\in I_\zeta = [(1-\zeta) R_1, (1+\zeta)R_1]$ where we smooth out the corner at $z=R_1$.

If we smooth $\widehat\lambda$ to $\lambda$, then we have $|\lambda_z|\leqslant |\widehat\lambda_z |$ on $I_\zeta$. In particular, we interpret $\widehat\lambda_z(R_1)$ to be $\max\left\{\lim\limits_{z\to R_1^-}\widehat\lambda_z(z), \lim\limits_{z\to R_1^+}\widehat\lambda_z(z)\right\}$. The estimates on $\lambda_z$ and $\widehat\lambda_z$ imply that on $I_\zeta$ we have $$|\lambda-\widehat\lambda|\leqslant(T-t_0)\zeta C(n,a,R_1).$$ By construction, if we define $\delta:=\inf_{I_\zeta} \{\lambda^+ - \widehat\lambda_0, \widehat\lambda_0 -  \lambda^- \}> 0$, then by choosing $t_0$ sufficiently close to $T$, we have $|\lambda-\widehat\lambda|\leqslant \delta/10$ and hence $\lambda^+ < \lambda < \lambda^- $.
\end{itemize}

\section{Convergence of $q$ in the proof of Lemma \ref{lem:type2}}\label{q_conv_appx}
Recall the definition of $q$ in \eqref{eq:def-q}.
\begin{itemize}
\item {\bf Step 1}: The linear bounds for $q$. 

\noindent {\it Claim:} for any fixed small $\eta>0$, we have 
$$q(\sqrt{\eta s}, s)=\sqrt{s}[y^{(0)}_\phi(\sqrt{\eta})+o(1)]$$ 
where $o(1)$ is as small as needed for sufficiently large $s$. 

\noindent {\it Proof of Claim:} by the definitions, 
$$q(z, s)=p_z(z, \tau)=e^\tau y_\phi\frac{\p \phi}{\p z}=e^{\tau/2}y_\phi=\sqrt{s}y_\phi(\phi, \tau),$$
and $\phi=z/\sqrt{s}$. So we conclude  
$$q(\sqrt{\eta s}, s)=\sqrt{s}y_\phi(\sqrt{\eta}, \tau)=\sqrt{s}[y^{(0)}_\phi(\sqrt{\eta})+o(1)]$$
where $o(1)$ is as small as needed for sufficiently large $\tau$ (or $s$) and the second equality makes use of the convexity of $y$ and its convergence to $y^{(0)}$ mentioned above. The claim is justified. 

\vspace{0.1in}

The evolution equation for $q(z, s)$ is 
\begin{equation}
\label{eq:P[q]}
\frac{\p q}{\p s}=-\frac{1}{s} z q_z+\frac{\p}{\p z}\left(\frac{q_z}{1+q^2} + \frac{n-1}{z} q \right).
\end{equation}
Let $Q_\lambda(z)=\tilde P'(\lambda z)$ and by \eqref{eq:tildeP}, we know that $Q_\lambda$ satisfies
$$\frac{(Q_\lambda)_z}{1+(Q_\lambda)^2} + \frac{n-1}{z}Q_\lambda=\lambda$$
and so we can eliminate the ``$\frac{\partial}{\partial_z}\left(\cdots\right)$'' term on the right hand side of the above evolution equation. Then in the same way as in \cite{AV97}, we can use $Q_\lambda$ to construct supersolutions and subsolutions for the above evolution equation in the region 
$$\Sigma_\eta=\{(z, s)~|~0\leqslant z\leqslant \sqrt{\eta s}, ~s\geqslant s_0\},$$ 
which is justified first on the boundary. From that, we conclude 
$$c_-z\leqslant q(z, s) \leqslant c_+z$$ 
which are the linear bounds for $q$. 

\item {\bf Step 2}: The convergence of $q(z, s)/z$.

For convenience of notations, we set 
$$G(z, s):=\frac{q(z, s)}{z}.$$
By the previous claim, we have
\begin{equation}
\begin{split}
G(\sqrt{\eta s}, s)
&= \frac{q(\sqrt{\eta s}, s)}{\sqrt{\eta s}} \\
&= \frac{1}{\sqrt{\eta}}[y^{(0)}_\phi(\sqrt{\eta})+o(1)] \\
&= \frac{1}{\sqrt{\eta}}\left(\frac{2a\sqrt{\eta}}{2n-2-\eta}+o(1)\right).
\end{split} \nonumber
\end{equation}
Thus as $s\to \infty$, 
$$G(\sqrt{\eta s}, s)\to \frac{2a}{2n-2-\eta},$$
and we conclude the special convergence 
$$\lim_{\eta\to 0}\lim_{s\to 0}G(\sqrt{\eta s}, s)=\frac{a}{n-1}$$
Now as for Lemma 7.3 in \cite{AV97}, we can justify the convergence more generally using the barrier argument. More precisely, we have that for any $\delta>0$, there exist $N$, $s_\delta$ and $\eta_\delta\in [0, \eta]$ such that 
$$\vline G(z, s)-\frac{a}{n-1}\vline\leqslant \delta$$ 
for $z\geqslant N$, $s\geqslant s_\delta$ and $z^2\leqslant \eta_\delta s$.

\item {\bf Step 3}: The convergence of $q$. 

We obtain the higher order estimates as in Lemma 7.4 of \cite{AV97}. Then using the standard sequence picking method, we find that as $l\to \infty$, $s_l\to \infty$ such that locally uniformly, we have 
$$q(z, s+s_l)\to q_\infty(z, s),$$
where the limit $q_\infty$ is a solution of
\begin{align}\label{eq:eq-q} 
\frac{\p q}{\p s} & =\frac{\p}{\p z}\left(\frac{q_z}{1+q^2} + \frac{n-1}{z} q \right).
\end{align}
in light of \eqref{eq:P[q]}. By the linear bounds and convergence from the previous steps, we have
$$c_-z\leqslant q_\infty(z, s)\leqslant c_+z, \quad\lim_{z\to\infty}\frac{q_\infty(z, s)}{z}=\frac{a}{n-1}.$$ 

Pick $\lambda_\infty$ such that 
$$\lim_{z\to\infty}\frac{Q_{\lambda_\infty}(z)}{z}=\frac{a}{n-1},$$
where $Q_{\lambda_\infty}(z)=\tilde P'(\lambda_\infty z)$ is an equilibrium solution of equation \eqref{eq:eq-q}. By \eqref{eq:asymptotics-tildeP}, we know $\lambda_\infty=a$. In fact, $q_\infty(z, s)=Q_{\lambda_\infty}(z)$ by the same argument in \cite[p.57--58]{AV97}. From this (uniqueness of sequential limit), we see that 
$$q(z, s)\to Q_{\lambda_\infty}(z), ~\text{as}~s\to\infty.$$
In other words, locally uniformly in $z$, 
$$p_z(z, \tau)\to \tilde P'(az), ~\text{as}~\tau\to\infty,$$
which concludes the proof of Lemma \ref{lem:type2}. 
\end{itemize}

\section{Discussion of Remark \ref{comp_rmk}}\label{comp_rmk_appx}
Despite the fact that the piecewise smooth upper barrier $\lambda^{+}$ and the piecewise smooth lower barrier $\lambda^{-}$ defined by \eqref{eq:upperbarrier} and \eqref{eq:lowerbarrier} respectively, are not smooth, the comparison principle (Proposition \ref{comparison}) applies to them. This is because, by Lemma \ref{patch}, the non-smooth points (i.e., ``corners) of $\lambda^{+}$ and $\lambda^{-}$ are unique for each $\tau$ and have the jumps of the first derivatives in favourable directions. It follows that the point of first contact between the subsolution $\lambda^-$ or supersolution $\lambda^+$ and the MCF solution (with appropriate boundary conditions as discussed in Proposition \ref{comparison}) is necessarily away from the corners, and thus these functions are smooth at this point of first contact. For completeness, we now provide more details.

Consider the case of $\lambda^+$ and $\lambda^-$, allowing for all scenarios. At the corners, it follows from Lemma \ref{patch} that for either $z$ or $\phi$, the spatial derivatives satisfy  
$$(\lambda^+_{int})'>(\lambda^+_{ext})', \quad (\lambda^-_{int})'<(\lambda^-_{ext})'.$$
Now let us consider all the possibilities of the first contact point where it is the corner for at least one of the functions involved. In the following discussion, we only make use of the spatial extremal property and so it is just the spatial minimum of $\lambda^+ - \lambda^-$ under consideration. We have the following possibilities.

\begin{itemize}
\item[(i)] It is the corner $p$ for $\lambda^+$ and before (in the sense that its $\phi$-coordinate or $z$-coordinate is closer to zero) the corner $q$ of $\lambda^-$. So it is the minimum of $\lambda^+_{int}-\lambda^-_{int}$ before $p$ and the minimum of $\lambda^+_{ext}-\lambda^-_{int}$ after $p$. Then at $p$, $(\lambda^+_{int})'\leqslant (\lambda^-_{int})'$ and $(\lambda^+_{ext})'\geqslant (\lambda^-_{int})'$, and so we have 
$$(\lambda^+_{ext})'\geqslant (\lambda^-_{int})'\geqslant (\lambda^+_{int})'>(\lambda^+_{ext})'$$
which is a contradiction. 

\item[(ii)] It is the corner $p$ for $\lambda^+$ and after (in the sense that its $\phi$-coordinate or $z$-coordinate is farther away from zero) the corner $q$ of $\lambda^-$. So it is the minimum of $\lambda^+_{int}-\lambda^-_{ext}$ before $p$ and the minimum of $\lambda^+_{ext}-\lambda^-_{ext}$ after $p$. Then at $p$, $(\lambda^+_{int})'\leqslant (\lambda^-_{ext})'$ and $(\lambda^+_{ext})'\geqslant (\lambda^-_{ext})'$. We have 
$$(\lambda^-_{ext})'\geqslant (\lambda^+_{int})'> (\lambda^+_{ext})'\geq(\lambda^-_{ext})'$$
which is a contradiction. 

\item[(iii)] It is the corner $q$ for $\lambda^-$ and before the corner $p$ of $\lambda^+$. So it is the minimum of $\lambda^+_{int}-\lambda^-_{int}$ before $q$ and the minimum of $\lambda^+_{int}-\lambda^-_{ext}$ after $q$. Then at $q$, $(\lambda^+_{int})'\leqslant (\lambda^-_{int})'$ and $(\lambda^+_{int})'\geqslant (\lambda^-_{ext})'$, and so we have 
$$(\lambda^-_{int})'\geqslant (\lambda^+_{int})'\geqslant (\lambda^-_{ext})'>(\lambda^-_{int})'$$
which is a contradiction. 

\item[(iv)] It is the corner $q$ for $\lambda^-$ and after the corner $p$ of $\lambda^+$. So it is the minimum of $\lambda^+_{ext}-\lambda^-_{int}$ before $q$ and the minimum of $\lambda^+_{ext}-\lambda^-_{ext}$ after $q$. Then at $q$, $(\lambda^+_{ext})'\leqslant (\lambda^-_{int})'$ and $(\lambda^+_{ext})'\geqslant (\lambda^-_{ext})'$, and so we have 
$$(\lambda^-_{int})'\geqslant (\lambda^+_{ext})'\geqslant (\lambda^-_{ext})'>(\lambda^-_{int})'$$
which is a contradiction. 

\item[(v)] It is the corner $p$ for $\lambda^+$ and also the corner of $\lambda^-$. So it is the minimum of $\lambda^+_{int}-\lambda^-_{int}$ before $p$ and the minimum of $\lambda^+_{ext}-\lambda^-_{ext}$ after $p$. Then at $p$, $(\lambda^+_{int})'\leqslant (\lambda^-_{int})'$ and $(\lambda^+_{ext})'\geqslant (\lambda^-_{ext})'$, and so we have 
$$(\lambda^-_{int})'\geqslant (\lambda^+_{int})'> (\lambda^+_{ext})'\geq(\lambda^-_{ext})'>(\lambda^-_{int})'$$
which is a contradiction. 	
\end{itemize}

So the first contact point has to be a smooth point. Furthermore, the spatial derivatives involved in the proof of Proposition \ref{comparison} are uniformly bounded in light of the explicit forms of $\lambda^\pm_{int}$ and $\lambda^\pm_{ext}$. Hence we know that Proposition \ref{comparison} is applicable for the piecewise smooth functions in our consideration for barriers. 

\section{Discussion of Remark \ref{ratio_R_rmk}}\label{ratio_R_appx}
We discuss the condition $\mathfrak{R}\leqslant C$, or equivalently, $C+Cu^2_x+u\cdot u_{xx}\geqslant 0$, on the initial hypersurface. Recall $u(x_0)=0$, $u_x>0$ and $u_{xx}<0$. We note that at the tip, although $\mathfrak{R}=1$, it does not follow that $1+u^2_x+u\cdot u_{xx}=0$ since the denominator of $\mathfrak{R}$ is $1+u^2_x$, which is infinity at the tip. 

The inverse of the function $r=u(x)$, i.e., the function $x=x(u)$, is a smooth even function. It is easy to see the leading term of $u(x)$ is $(x-x_0)^\eta$ for $\eta\in (0, \frac{1}{2}]$, and so the functions $\kappa_1. \cdots, \kappa_n$ have the same leading term $(x-x_0)^{1-2\eta}$, and so $\mathfrak{R}$ is continuous up to $x_0$. In fact, $\eta=1/2$, for otherwise the principal curvatures at the tip $x=x_0$ are all zero, contradicting that the hypersurface is strictly convex.

Near the tip, this can be reduced to a condition on the expansion of $u(x)$ or $x(u)$. Namely, after translation in the $x$-direction so that the tip occurs at $x=0$, we can let
\begin{align*}
u &= \alpha x^{\frac{1}{2}}+\beta x^{\frac{3}{2}}+O(x^{\frac{5}{2}}).
\end{align*}
Then we have 
$$u_x=\frac{\alpha}{2}x^{-\frac{1}{2}}+\frac{3\beta}{2}x^{\frac{1}{2}}+O(x^{\frac{3}{2}}), \quad u_{xx}=-\frac{\alpha}{4}x^{-\frac{3}{2}}+\frac{3\alpha}{4}x^{-\frac{1}{2}}+O(x^{\frac{1}{2}}),$$
and so
$$1+u^2_x+u\cdot u_{xx}=1+2\alpha \beta + O(x).$$
Hence, we impose the relation $1+2\alpha\beta=0$ so that $1+u^2_x+u\cdot u_{xx}=0$, or equivalently, $\mathfrak{R}=1$, at the tip where $x=0$. For example, if we let $\alpha=1$, then $1+2\alpha\beta=0$ implies that $\beta=-1/2$, and so we write out the next term in the expansion of $u(x)$:
$$u=x^{\frac{1}{2}}-\frac{1}{2}x^{\frac{3}{2}}+\gamma x^{\frac{5}{2}}+O(x^{\frac{7}{2}}),$$
whence it follows that
\begin{align*}
u_x&=\frac{1}{2}x^{-\frac{1}{2}}-\frac{3}{4}x^{\frac{1}{2}}+\frac{5\gamma}{2}x^{\frac{3}{2}}+O(x^{\frac{5}{2}}),\\ u_{xx}&=-\frac{1}{4}x^{-\frac{3}{2}}-\frac{3}{8}x^{-\frac{1}{2}}+\frac{15\gamma}{4}x^{\frac{1}{2}}+O(x^{\frac{3}{2}}).
\end{align*}
Then we obtain
\begin{align*}
C+Cu^2_x+u\cdot u_{xx}=\frac{C-1}{4x} + \frac{C-1}{4} + \frac{3+56\gamma+C(9+40\gamma)}{16}x + O\left(x^2\right).
\end{align*}
Hence we just need to require 
\begin{align}\label{eq:requirement}
C\geqslant 1 \quad\text{and}\quad \frac{3+56\gamma+C(9+40\gamma)}{16}\geqslant 6\gamma+\frac{3}{4}>0
\end{align}
to guarantee that $\mathfrak{R}\leqslant C$ near the tip, which is an open condition.  

Towards the cylindrical end, both $u_x$ and $u_{xx}$ tend to $0$ while $u$ stays bounded, so $C+Cu^2_x+u\cdot u_{xx}$ tends to $C$ and $R$ tends to $0$. In any case, the requirement \eqref{eq:requirement} is not restrictive on the asymptotic cylindrical ends.


\bibliography{mcf-type2_bib}
\bibliographystyle{alpha}

\end{document}